\documentclass[12pt,reqno]{amsart}
\usepackage[english]{babel}
\usepackage{a4wide}
\usepackage[T1]{fontenc}
\usepackage{amssymb,amsmath,amsthm,latexsym}
\usepackage{mathrsfs}
\usepackage[usenames,dvipsnames]{color}
\usepackage{graphicx}
\usepackage{mdwlist}
\usepackage{enumerate}
\usepackage{mathtools,dsfont,wasysym}
\usepackage{stmaryrd}
\usepackage{hyperref}
\usepackage{multicol}
\hypersetup{colorlinks=true,  linkcolor=blue, citecolor=red, urlcolor=cyan}

\newtheorem{theorem}{Theorem}[section]
\newtheorem{lemma}[theorem]{Lemma}

\newtheorem{proposition}[theorem]{Proposition}

\newtheorem{corollary}[theorem]{Corollary}

\theoremstyle{definition}

\newtheorem{example}[theorem]{Example}

\theoremstyle{remark}
\newtheorem{remark}[theorem]{Remark}
\numberwithin{equation}{section}

\newcommand{\R}{\mathbb{R}}
\newcommand{\C}{\mathbb{C}}
\newcommand{\N}{\mathbb{N}}
\newcommand{\K}{\mathbb{K}}
\newcommand{\Q}{\mathbb{Q}}

\newcommand{\co}{\mathfrak{c}}

\newcommand{\nn}[1]{{\left\vert\kern-0.25ex\left\vert\kern-0.25ex\left\vert #1 
		\right\vert\kern-0.25ex\right\vert\kern-0.25ex\right\vert}}
\renewcommand{\geq}{\geqslant}
\renewcommand{\leq}{\leqslant}

\newcommand{\norm}[1]{\left\Vert#1\right\Vert}

\newcounter{smallromans}

	{\end{list}}

\makeatletter
\renewcommand{\tocsection}[3]{%
	\indentlabel{\@ifnotempty{#2}{\bfseries\ignorespaces#1 #2\quad}}\bfseries#3}
\renewcommand{\tocsubsection}[3]{%
	\indentlabel{\@ifnotempty{#2}{\ignorespaces#1 #2\quad}}#3}

\newcommand\@dotsep{4.5}
\def\@tocline#1#2#3#4#5#6#7{\relax
	\ifnum #1>\c@tocdepth 
	\else
	\par \addpenalty\@secpenalty\addvspace{#2}%
	\begingroup \hyphenpenalty\@M
	\@ifempty{#4}{%
		\@tempdima\csname r@tocindent\number#1\endcsname\relax
	}{%
		\@tempdima#4\relax
	}%
	\parindent\z@ \leftskip#3\relax \advance\leftskip\@tempdima\relax
	\rightskip\@pnumwidth plus1em \parfillskip-\@pnumwidth
	#5\leavevmode\hskip-\@tempdima{#6}\nobreak
	\leaders\hbox{$\m@th\mkern \@dotsep mu\hbox{.}\mkern \@dotsep mu$}\hfill
	\nobreak
	\hbox to\@pnumwidth{\@tocpagenum{\ifnum#1=1\bfseries\fi#7}}\par
	\nobreak
	\endgroup
	\fi}
\AtBeginDocument{%
	\expandafter\renewcommand\csname r@tocindent0\endcsname{0pt}
}
\def\l@subsection{\@tocline{2}{0pt}{2.5pc}{5pc}{}}
\makeatother

\usepackage{todonotes}

\begin{document}

	\title[Searching for linear structures in the failure of Stone-Weierstrass]{Searching for linear structures in the failure of the Stone-Weierstrass theorem}

		\author[Caballer]{Marc  Caballer}
	\address[Caballer]{Facultad de Ciencias Matemáticas, Universidad de Valencia, Doctor Moliner 50, 46100 Burjasot (Valencia), Spain. \newline
		\href{http://orcid.org/0000-0000-0000-0000}{ORCID: \texttt{0009-0003-3893-0131} }}
	\email{\texttt{marcago5@alumni.uv.es}}

 	\author[Dantas]{Sheldon Dantas}
	\address[Dantas]{Departamento de Análisis Matemático, Facultad de Ciencias Matemáticas, Universidad de Valencia, Doctor Moliner 50, 46100 Burjasot (Valencia), Spain. \newline
		\href{http://orcid.org/0000-0001-8117-3760}{ORCID: \texttt{0000-0001-8117-3760} } }
	\email{\texttt{sheldon.dantas@uv.es}}
	
		\author[Rodríguez-Vidanes]{Daniel L. Rodríguez-Vidanes}
	\address[Rodríguez-Vidanes]{Grupo de Investigación: Análisis Matemático y Aplicaciones, Departamento de Matemática Aplicada a la Ingeniería Industrial, Escuela Técnica Superior de Ingeniería y Diseño Industrial, Ronda de Valencia 3, Universidad Politécnica de Madrid, Madrid, 28012, Spain. \newline
		\href{http://orcid.org/0000-0000-0000-0000}{ORCID: \texttt{0000-0003-2240-2855} }}
	\email{\texttt{dl.rodriguez.vidanes@upm.es}}
			
	\begin{abstract} We investigate the failure of the Stone-Weierstrass theorem focusing on the existence of large dimensional vector spaces within the set $\mathcal{C}(L, \K) \setminus \overline{\mathcal{A}}$, where $L$ is a compact Hausdorff space and $\mathcal{A}$ is a self-adjoint subalgebra of $\mathcal{C}(L, \K)$ that vanishes nowhere on $L$ but does not necessarily separate the points of $L$. We address the problem of finding the precise codimension of $\overline{\mathcal{A}}$ in a broad setting, which allows us to describe the lineability of $\mathcal{C}(L, \K) \setminus \overline{\mathcal{A}}$ in detail. Our analysis yields both affirmative and negative results regarding the lineability of this set. Furthermore, we also study the set $(\mathcal{C}(\partial{D}, \C) \setminus \overline{\text{Pol}(\partial{D})}) \cup \{0\}$, where $\text{Pol}(\partial{D})$ is the set of all complex polynomials in one variable restricted to the boundary of the unit disk. Recent lineability properties are also taken into account.

	\end{abstract}

	\thanks{}

	\subjclass[2020]{Primary 46B87; Secondary 15A03, 41A10, 46J10, 46E15}
	\keywords{Lineability; spaceability; dense-lineability; algebrability; Stone-Weierstrass theorem; self-adjoint algebras}
	
	\maketitle
	
	
	\thispagestyle{plain}

\section{Introduction}

\subsection{Motivation} The Stone-Weierstrass theorem states that if $L$ is a compact Hausdorff space and $\mathcal{A}$ is a subalgebra of $\mathcal{C}(L, \R)$ that separates the points of $L$, then $\overline{\mathcal{A}} = \mathcal{C}(L, \R)$. However, this result does not extend to the complex field. In fact, as shown in works dating back to the 1800s, not every compact subset of $\C$ satisfies this theorem. A classic counterexample is when $L \subseteq \C$ is the unit circle and $\mathcal{A}$ is the subalgebra of $\mathcal{C}(L, \C)$ consisting of all complex polynomials restricted to $L$; in this case, $\mathcal{A}$ is not dense in $\mathcal{C}(L, \C)$, as $\overline{\mathcal{A}}$ does not contain the map $z \mapsto \bar{z}$ (see, for instance, \cite[Example 3.120]{S}).

In this paper, we investigate the linear structures within the set $\mathcal{C}(L, \K) \setminus \overline{\mathcal{A}}$ (for both $\K = \R$ and $\K = \C$) whenever it is non-empty, that is, when the Stone-Weierstrass theorem does not hold for some compact $L$. It is well-known that $\mathcal{C}(L, \K) \setminus \overline{\mathcal{A}}$ contains a vector space of dimension $\kappa$ if and only if $\kappa \leq \text{codim}(\overline{\mathcal{A}})$.
However, determining the exact value of $\text{codim}(\overline{\mathcal{A}})$ does not seem to be straightforward. Therefore, our primary goal is to precisely compute $\text{codim}(\overline{\mathcal{A}})$ by relating it to more accessible and manageable concepts within a broad framework, allowing for the exact determination of $\text{codim}(\overline{\mathcal{A}})$ in various examples.

The study of these problems, known as lineability, has been an active area of research since its introduction by V. I. Gurariy in the early 2000s, with significant developments in various areas of mathematics (see the monograph \cite{ABPS} and papers \cite{BPS, FGK, GQ, LRS, LM, Rm, SS}, among others).

Concerning the Stone-Weierstrass theorem, it is worth mentioning that there has been works on identifying conditions under which a compact Hausdorff space $L$ ensures that every subalgebra of $\mathcal{C}(L, \C)$ that separates points and contains the constant functions is dense in $\mathcal{C}(L, \C)$. This is known as the complex Stone-Weierstrass property (CSWP, for short). For example, W. Rudin showed that if $L$ contains a copy of the Cantor set, then $L$ does not satisfy the CSWP; conversely, if $L$ is scattered, it does satisfy the CSWP. For more on this line of research, see \cite{HK, K, Rudin}.

\subsection{Description of our results} 
Let us briefly describe the contents of the present paper. In Section~\ref{preliminaries}, we show the main concepts, notation and tools used throughout this work. The main lineability concepts used here are defined, the tools needed in order to tackle the lineability properties in the context of the failure of Stone-Weierstrass theorem are established and the set-theoretical background is presented. In particular, an equivalence relation $\sim_\mathcal{A}$ is defined on a subalgebra $\mathcal{A}$ of $\mathcal{C}(L, \K)$, which will play a crucial role in the main results, and we introduce the order of $\sim_{\mathcal A}$, denoted by $\textup{ord}(\sim_{\mathcal A})$, which is a cardinal number that depends on $\sim_\mathcal{A}$. In Section~\ref{main}, we state and prove our main results.  It is divided into three subsections. In the first subsection, we show a negative lineability result by proving that $\mathcal{C}(L, \K) \setminus \overline{\mathcal A}$ is $\textup{ord}(\sim_{\mathcal A})$-lineable but not $(\textup{ord}(\sim_{\mathcal A})+1)$-lineable provided that the order of $\sim_{\mathcal A}$ is finite (see Theorem \ref{nonlineablereal}). This provides a consequence on latticeability as well (see Corollary \ref{latticeable}) In the second subsection, we tackle the case when the order of $\sim_{\mathcal A}$ is infinite (see Proposition \ref{linnumber}). We obtain a general result that goes deeper into the failure of the Stone-Weierstrass theorem and depends on a notion that we introduce, which is called the Stone-Weierstrass character of a compact Hausdorff space $L$. We analyze some of its properties by showing its relation to the topological weight of $L$ and generalize the previous lineability results. To be more precise, the notion of the Stone-Weierstrass character provides valuable information regarding lineability properties of $\mathcal{C}(L,\K)\setminus\overline{\mathcal{A}}$ even in more general sets (see Theorem \ref{thmcharacter} and Proposition \ref{prop-sw(S)}). Using Theorem \ref{thmcharacter}, we are able to provide more negative lineability results (see Proposition \ref{prpcont} and Corollary \ref{cor:prpcont}). Finally, in Section \ref{complex-plane}, the failure of the classical Stone-Weierstrass theorem is analyzed in the context of lineability improving and deepening already known results. Recent notions of lineability are also studied in this context.

\section{Preliminaries}\label{preliminaries}

\subsection{Basic notation and first tools} In this section, we introduce the notation, definitions and all the results that we need throughout the paper. In fact, our intention here is to make this manuscript entirely accessible with the intention that the reader avoids jumping into many different references. Let us denote by $\K$ either the set of real numbers $\R$ or the set of complex numbers $\C$. All the topological spaces $L$ considered in this paper are compact Hausdorff unless otherwise stated. The Banach space $\mathcal C(L, \K)$ stands for all continuous functions defined on $L$ with the standard sup-norm. Let $\mathcal{A}$ be a set of functions on $L$. We say that $\mathcal{A}$ {\it separates the points of $L$} if, whenever $x\not=y$ in $L$, there is a function $f \in \mathcal{A}$ such that $f(x) \not= f(y)$. We say that $\mathcal{A}$ {\it vanishes nowhere on $L$} if, for every $x \in L$, there exists $f \in \mathcal{A}$ such that $f(x) \not= 0$. We say that $\mathcal{A}$ is {\it self-adjoint} whenever $f \in \mathcal{A}$ we have that $\overline{f} \in \mathcal{A}$. 
If $\rho$ is an equivalent relation in $L \times L$, we will write $\rho \subseteq L \times L$ to represent the elements of $L \times L$ which are related by $\rho$. By a subalgebra we understand a linear subspace closed under product.

We will be using basic concepts and notations from Set Theory found, for instance, in \cite{C, J}. Ordinal numbers will be identified with the set of their predecessors and cardinal numbers with the initial ordinals. Given a set $A$, the cardinality of $A$ will be denoted by $|A|$. We denote by $\aleph_0$, $\aleph_1$ and $\co$ the first infinite cardinal, the second infinite cardinal and the cardinality of the continuum, respectively.

\vspace{0.2cm}
We consider the map 
    \begin{align*}
        \Delta:\mathcal{C}(L,\K) & \to\mathcal{C}(L\times L,\K) \\
        f & \mapsto \Delta(f)(x,y) := f(x)-f(y) 
    \end{align*}
which sends a continuous function into the difference with itself. Clearly, $\Delta$ is a continuous linear operator. Given a family of functions $F\subseteq\mathcal{C}(L,\K)$, we consider an equivalence relation $\sim_{F}$ on $L$ defined by 
\begin{equation*} 
    x\sim_{F}y \iff f(x)=f(y) \text{ for all } f \in F.
\end{equation*} 
Let us notice that $F$ separates points if and only if $\sim_F$ coincides with the diagonal 
\begin{equation*} 
D_L:=\{(x,x):x\in L\}. 
\end{equation*}
Moreover, for any $F\subseteq\mathcal C(L,\K)$, the equivalence relation $\sim_F\subseteq L\times L$ is closed as a subspace since
\begin{equation*} 
    \sim_F=\bigcap_{f\in F}\{(x,y)\in L\times L:f(x)=f(y)\}=\bigcap_{f\in F}\Delta(f)^{-1}[\{0\}].
\end{equation*} 
So, from now on, $\rho\subseteq L\times L$ will denote a closed equivalence relation which, in turn, will be compact Hausdorff since $L\times L$ is compact Hausdorff. Notice that $L/\rho$ will always be compact but not necessarily Hausdorff.

For a closed subspace $S\subseteq L$, the restriction map $\text{res}_S^L : \mathcal{C}(L,\K) \to \mathcal{C}(S,\K)$, which sends any $f \in \mathcal{C}(L, \K)$ to $f|_S$ will be of interest for us. Let us notice that by using Tietze's extension theorem, the function $\text{res}_S^L$ is always surjective. With this in mind, we then consider the map $\Delta_\rho:=\text{res}_{\rho}^{L\times L}\circ\Delta:\mathcal{C}(L,\K)\to\mathcal{C}(\rho,\K)$ and define $\mathcal{C}_\rho(L,\K):=\text{Ker}(\Delta_\rho)$ and $\mathcal{C}_L(\rho,\K):=\text{Im}(\Delta_\rho)$ Hence, we have that 
\begin{equation*} 
    \mathcal{C}_\rho(L,\K)=\{f\in\mathcal{C}(L,\K):x\rho y\Rightarrow f(x)=f(y)\}.
\end{equation*}
This means that, for any $F\subseteq\mathcal{C}(L,\K)$, we have $F\subseteq\mathcal{C}_{\sim_F}(L,\K)$. On the other hand, let us observe that $\mathcal{C}_\rho(L,\K)$ can be seen as an image, that is, the map $\Psi:\mathcal{C}(L/\rho,\K)\to\mathcal{C}(L,\K)$ defined by $f \mapsto f\circ\pi$, where $\pi$ is the quotient map, is an isometric homomorphism between Banach algebras and the range of $\Psi$ is precisely $\mathcal{C}_\rho(L,\K)$. 
In other words, $\mathcal{C}_\rho(L,\K)$ is isometrically isomorphic to $\mathcal{C}(L/\rho,\K)$. Also note that $\mathcal{C}_\rho(L,\K)$ is closed and so $\mathcal{C}_L(\rho,\K)\cong\mathcal{C}(L,\K)/\mathcal{C}_L(\rho,\K)$ is a Banach space.

\vspace{0.5cm}
With all this notation in mind, let us recall the following Stone-Weierstrass type theorem, which can be found, for instance, in \cite[Exercise~10, page 158]{S}. Since we will be using Lemma~\ref{lemrealequi} several times throughout the paper, we provide its proof for the sake of completeness. Notice that it holds true for both real and complex cases. 

\begin{lemma} \label{lemrealequi} Let $\mathcal{A}\subseteq\mathcal{C}(L,\K)$ be a self-adjoint algebra over $\K$ that vanishes nowhere on $L$. Then $\mathcal{A}$ is dense in $\mathcal{C}_{\sim_{\mathcal{A}}}(L,\K)$.
\end{lemma}

\begin{proof} Since $\mathcal{C}_{\sim_{\mathcal{A}}}(L,\K)$ is closed, we already have that $\overline {\mathcal{A}}\subseteq\mathcal{C}_{\sim_{\mathcal{A}}}(L,\K)$. In order to show the other inclusion, we can proceed as follows. As $\Psi:\mathcal{C}(L/\sim_{\mathcal{A}},\K) \to \mathcal{C}_{\sim_{\mathcal{A}}}(L,\K)$ is an isomorphism between Banach algebras, we know that $\Psi^{-1}[\mathcal A]\subseteq\mathcal{C}(L/\sim_{\mathcal{A}},\K)$ is a subalgebra. 
Moreover, given $[x],[y]\in L /\sim_{\mathcal{A}}$ distinct, we have that $x\not\sim_{\mathcal{A}}y$, so there is an $f\in\mathcal{A}$ such that $f(x)\neq f(y)$.
So, there exists a $g\in\Psi^{^{-1}}[\mathcal A]$ with $g([x])\neq g([y])$, that is, $\Psi^{-1}[\mathcal A]$ separates points. Assume now that $\Psi^{^{-1}}[\mathcal A]$ vanishes at some point $[x]\in L/\sim_{\mathcal A}$, i.e., $g([x])=0$ for every $g\in \Psi^{^{-1}}[A]$.
Then, since the range of $\Psi$ is $\mathcal{C}_{\sim_{\mathcal{A}}}(L,\K)$, $\mathcal A\subseteq \mathcal{C}_{\sim_{\mathcal{A}}}(L,\K)$ and $\Psi (g)(x)=(g \circ \pi)(x)=g([x])=0$ for every $g\in \Psi^{^{-1}}[\mathcal A]$, it would yield that $f(x)=0$ for every $f\in \mathcal A$. Therefore, $\mathcal A$ would vanish at $x\in L$, which is a contradiction. Now, let us show that $\Psi^{-1}[\mathcal A]$ is self-adjoint. Take $f\in \Psi^{-1}[\mathcal A]$. Then, $\Psi(f)=f\circ \pi \in \mathcal A$ and, by the self-adjointness of $\mathcal A$, we have 
$\mathcal A \ni \overline{f\circ \pi} = \bar{f} \circ \pi = \Psi (\bar{f})$. So, $\bar{f}\in \Psi^{-1}[\mathcal A]$, as needed. Then, by the Stone-Weierstrass theorem, $\Psi^{-1}[\mathcal{A}]$ is dense in $\mathcal{C}(L/\sim_{\mathcal{A}},\K)$ and so $\mathcal{A}$ is dense in $\mathcal{C}_{\sim_{\mathcal{A}}}(L,\K)$. This completes the proof. 
\end{proof}

Let us point out that, in view of Lemma~\ref{lemrealequi}, since $\mathcal{C}(L,\K)\setminus\overline{\mathcal A}=\mathcal{C}(L,\K)\setminus\mathcal{C}_{\sim_{\mathcal A}}(L,\K)$, where $\mathcal A \subseteq \mathcal{C}(L,\K)$ is a self-adjoint algebra that vanishes nowhere on $L$, we can stretch our study to $\mathcal{C}(L,\K)\setminus\mathcal{C}_{\rho}(L,\K)$ for closed equivalence relations $\rho\subseteq L\times L$. 

\subsection{Lineability concepts} \label{lineabilitynotation} In this short section, we define the main concepts related to lineability that will appear in this work. 

\vspace{0.2cm}
Let $V$ be a vector space, $M$ a subset of $V$ and $\kappa$ a cardinal number. Let us denote by $\dim(V)$ the algebraic dimension of $V$ over the field $\mathbb K$. We say that $M$ is $\kappa$-{\it lineable} if $M \cup \{0\}$ contains a vector subspace of $V$ of dimension $\kappa$.
In addition, if $V$ is a topological vector space, then we say that $M$ is
		\begin{itemize}
			\setlength\itemsep{0.3em}
		\item[(i)] $\kappa$-{\it spaceable} if $M \cup \{0\}$ contains a closed vector subspace of $V$ of dimension $\mu$. 
		\item[(ii)] {\it maximal-spaceable} if $M$ is $\dim(V)$-spaceable.
        \item[(iii)] {\it $\kappa$-dense-lineable} if $M\cup \{0\}$ contains a dense vector subspace $V$ of dimension $\kappa$.
		\end{itemize}
  

\vspace{0.2cm}
V.V. Fávaro, D. Pellegrino, and D. Tomaz introduced the following more restrictive notion of spaceability called $(\alpha, \beta)$-spaceability (see \cite{FPT}). Let $\alpha\leq \beta$ be two cardinal numbers. 
We say that $M$ is $(\alpha,\beta)$-{\it spaceable} if $M$ is $\alpha$-lineable and for each $\alpha$-dimensional vector subspace $V_\alpha$ of $V$ with $V_\alpha \subseteq M\cup \{0\}$, there is a closed $\beta$-dimensional subspace $V_\beta$ of $V$ such that $V_\alpha \subseteq V_\beta \subseteq M\cup \{0\}$. We see that $(\alpha,\alpha)$-spaceability implies $\alpha$-spaceability, but the converse is not true in general. Indeed, it is well-known that the set $L_p[0,1] \setminus \bigcup_{q \in (1,\infty)} L_q [0,1]$ is $\mathfrak c$-spaceable for any $p>0$ (see \cite{BFPS}) and, as an immediate consequence of \cite[Corollary~2.4]{FPRR}, the set $L_p[0,1] \setminus \bigcup_{q \in (1,\infty)} L_q [0,1]$ is not $(\mathfrak c,\mathfrak c)$-spaceable. It is worth mentioning that Araújo, Barbosa, Raposo Jr. and Ribeiro proved in fact that 
$L_p[0,1] \setminus \bigcup_{q \in (1,\infty)} L_q [0,1]$ is $(\alpha,\mathfrak c)$-spaceable if and only if $\alpha < \aleph_0$ (see \cite[Theorem~3]{ABRR}).
It is clear that $\alpha_1$-spaceability implies $\alpha_2$-spaceability if $\alpha_2\leq \alpha_1$, but the converse is not true in general.
In the context of $(\alpha,\beta)$-spaceability, $(\alpha_1,\beta)$-spaceability does not imply $(\alpha_2,\beta)$-spaceability and viceversa. 

\vspace{0.2cm}

Now, let us come back to our main problem. The notation is the one introduced in the previous section. Recall that $\mathcal{C}(L,\K)\setminus\mathcal{C}_{\rho}(L,\K)$ is $\kappa$-lineable if and only if $\kappa\leq\text{codim}(\mathcal{C}_{\rho}(L,\K))$. Furthermore, by a result of Kitson and Timoney (see \cite[Theorem~2.2]{KT} and also \cite[p. 12]{W} by Wilansky), we have that $\mathcal{C}(L,\K)\setminus\mathcal{C}_{\rho}(L,\K)$ is spaceable if and only if $\text{codim}(\mathcal{C}_{\rho}(L,\K))$ is infinite as $\mathcal{C}(L,\K)$ is a Fréchet space and $\mathcal{C}_{\rho}(L,\K)$ is a closed subspace. In fact, by a more recent result due to  P. Leonetti, T. Russo, and J. Somaglia (see \cite[Corollary~3.2]{LRS} or \cite[Lemma~3.1]{LRS}), if $\text{dens}(\mathcal{C}(L,\K)) \leq \text{codim}(\mathcal{C}_{\rho}(L,\K))$, then $\mathcal{C}(L,\K)\setminus\mathcal{C}_{\rho}(L,\K)$ is $\text{dens}(\mathcal{C}(L,\K))$-dense-lineable.
Now notice that
\begin{equation*} 
    \text{codim}(\mathcal{C}_\rho(L,\K))=\text{codim}(\text{Ker}(\Delta_\rho))=\text{dim}(\text{Im}(\Delta_\rho))=\text{dim}(\mathcal{C}_L(\rho,\K)).
\end{equation*} 
So, in view of the above arguments, the study of the $\kappa$-lineability of $\mathcal{C}(L,\K)\setminus\overline{\mathcal A}$ depends solely on the dimension of the Banach space $\mathcal{C}_L(\sim_{\mathcal A},\K)$.
\\

\subsection{Representatives of an equivalence relation} We say that $R\subseteq L$ is a {\it set of representatives of $\rho$} if there is a map $\xi_R : L/\rho\to L$ such that $\xi_R[L/\rho]=R$ and $\pi\circ\xi_R=\text{id}_{L/\rho}$. 
That is, $R$ is a set of representatives if it is a set that contains one and only one element of each equivalence class of $\rho$. Moreover, note that the map $\xi_R$ is unique for a given set of representatives $R$. Therefore, we will call {\it $\xi_R$ the associated map of $R$}. It is clear that, if $R_1$ and $R_2$ are two such sets, then $|R_1|=|L/\rho|=|R_2|$. In fact, we have also that $|L\setminus R_1|=|L\setminus R_2|$ as shown by the following set theoretical lemma. The symbol $\sqcup$ below stands for a disjoint union of sets, which plays an important role in the proof of the lemma.

\begin{lemma}\label{invariant} Let $R_1,R_2\subseteq L$ be two sets of representatives of $\rho$. Then $|L\setminus R_1|=|L \setminus R_2|$.
\end{lemma}

\begin{proof}
Let $\xi_1,\xi_2:L/\rho\to L$ be the associated maps of $R_1$ and $R_2$, respectively.
Then, consider the map $\xi_1\circ\pi:L\to L$ and denote $R_2'=R_1\setminus R_2,R_1'=R_2\setminus R_1$ and $I=R_1\cap R_2$. Now note that
    $$
    (\xi_1\circ\pi)[R_2]=\xi_1[\pi[R_2]]=\xi_1[L/\rho]=R_1
    $$
and that $(\xi_1\circ\pi)\vert_{R_2}$ is injective since, for any $x,y\in R_2$ such that $(\xi_1\circ\pi)(x)=(\xi_1\circ\pi)(y)$, we have
    $$
    \xi_1([x])=\xi_1([y])\Rightarrow[x]=[y]\Rightarrow x\sim y\Rightarrow x=y
    $$
as $\pi\circ\xi_1=\text{id}_{L/\rho}$ implies that $\xi_1$ is injective and class representatives in $R_2$ are unique. 
Hence, $\xi_1\circ\pi:R_2\to R_1$ defines a bijection with $(\xi_1\circ\pi)[I]=I$ since $(\xi_1\circ\pi)(x)=x$ for every $x\in R_1$ and, in particular, for every $x\in I$. This means that $\xi_1\circ\pi:R_1'\to R_2'$ is a bijection as $R_2=R_1'\sqcup I$ and $R_1=R_2'\sqcup I$. Then, since 
    \begin{equation}\label{identity}
        L\setminus R_i=R_i'\sqcup(L\setminus (R_1\cup R_2))
    \end{equation} 
for every $i\in \{1,2\}$, we define $\Xi: L \setminus R_1 \to L \setminus R_2$ by
    $$
    \Xi(x)=
    \begin{cases}(\xi_1\circ\pi)(x)&\text{if }x\in R_1',
    \\ x&\text{if }x\in L\setminus (R_1\cup R_2).
    \end{cases}
    $$
Since $\Xi|_{R_1'}:R_1'\to R_2'$ and $\Xi|_{L\setminus (R_1\cup R_2)} : L\setminus (R_1\cup R_2)\to L\setminus (R_1\cup R_2)$ are bijections, it follows from \eqref{identity} that $\Xi: L \setminus R_1 \to L\setminus R_2$ is a bijection and we are done.
\end{proof}

As the cardinal number $|L\setminus R|$ does not depend on the set of representatives $R$ of $\rho$ (by Lemma~\ref{invariant}), we call this cardinal invariant of $\rho$ the \emph{order} of $\rho$ and denote it by $\text{ord}(\rho)$. In fact, we can compute the order of $\rho$ by using the following formula
    \begin{equation}\label{suboptimarorder}
    \text{ord}(\rho)=|L\setminus R|=\left|\bigsqcup_{r\in R}[r]\setminus R\right|=\sum_{r\in R}|[r]\setminus\{r\}|.
    \end{equation}

We say that $\rho$ is \emph{finite} whenever it has finite order and \emph{infinite} otherwise. 
It is not hard to see that $\rho$ is finite if and only if the classes are finite and $\{[x]\in L/\rho:|[x]|>1\}$ is finite. Therefore, if $\rho$ is finite, then we can describe \eqref{suboptimarorder} as
\begin{equation*}
    \text{ord}(\rho) = \sum_{[x]\in L/\rho} (|[x]|-1)
\end{equation*} 
which is a finite sum of finite numbers. By infinite cardinal arithmetic (see, for instance, \cite[Chapter~5]{J}), if $\rho$ is infinite, then this identity holds true by identifying $\kappa-1:=\kappa$ for infinite cardinals $\kappa$. Let us notice that $\text{ord}(\rho)$ describes the number of identifications needed to fully describe $\rho$.

\begin{example} If $\sim$ is an equivalence relation of $L$ given by $x_1\sim x_2\sim x_3$ and $y_1\sim y_2$ provided that $x_1$, $x_2$, $x_3$, $y_1$ and $y_2$ are distinct and any other element of $L$ is related by $\sim$ only to itself, then $\sim$ is finite since for all $x\in L$ we have that $|[x]|\leq 3$ and also $|\{ [x] \in L/\sim\ :|[x]|>1\}|=2$ and its order is
\begin{equation*} 
    \text{ord}(\sim)=(|\{x_1,x_2,x_3\}|-1)+(|\{y_1,y_2\}|-1)=2+1=3.
\end{equation*}
\end{example} 

\vspace{0.4cm}
The following lemma is of crucial importance and interest on its own.
It provides an upper bound for the dimension of vector spaces inside the set $\mathcal C(L,\K) \setminus \mathcal C_\rho(L,\K)$ and, more specifically, for $\mathcal C(L,\K) \setminus \overline{\mathcal A}$ by Lemma~\ref{lemrealequi}. To be more precise, it shows that 
\begin{equation*} 
    \text{dim}(\mathcal C_L(\rho,\K))\leq \text{dim}(\K^{\text{ord}(\rho)}).
\end{equation*} 

\begin{lemma}\label{orderinjection}
There exists a linear injection from $\mathcal{C}_L(\rho,\K)$ into $\K^{\textup{ord}(\rho)}$.
\end{lemma}

\begin{proof}
Take $R$ to be any set of representatives of $\rho$ and consider the map
    \begin{align*}
        J:\mathcal{C}_L(\rho,\K) & \to\K^{L\setminus R} \\
        f & \mapsto J(f)(x)=f(x,\xi_R([x]))
    \end{align*}
where $\xi_R$ is the associated map of $R$.
The map $J$ is clearly linear. Let us prove that it has trivial kernel. 
Note that, if $f\in\mathcal{C}_L(\rho,\K)$ is such that $J(f)\equiv 0$, i.e., $f(x,\xi_R([x]))=0$ for every $x\in L\setminus R$, then we have that $f(x,\xi_R([x]))=0$ for every $x\in R$ since $f\in\mathcal{C}_L(\rho,\K)$ implies $f\vert_{D_L}\equiv 0$. 
In turn, given $(x,y)\in\rho$, we have
    $$
    f(x,y)=f(x,\xi_R([x]))-f(y,\xi_R([x]))=f(x,\xi_R([x]))-f(y,\xi_R([y]))=0-0=0.
    $$
Since there is a linear bijection $L$ from $\K^{L\setminus R}$ to $\K^{|L\setminus R|}=\K^{\text{ord}(\rho)}$, we have that the function $L\circ J$ is the one we were looking for.
\end{proof}

\section{Main Results}\label{main}

We divide this section into three subsections. We study first what happens when the relation $\sim_{\mathcal{A}}$ is finite and obtain a negative lineability result (see Theorem \ref{nonlineablereal} below). Then we move to the case when $\sim_{\mathcal{A}}$ is infinite. Finally, in the last subsection, we prove that $\mathcal{C}(\partial D,\C)\setminus\overline{\textup{Pol}(\partial D)}$ contains an isometric copy of $\text{Hol}(\partial D)$, where $\textup{Pol}(\partial D)$ denotes the set of all complex polynomials in one variable restricted to the boundary $\partial D$ of the unit disc.

\subsection{On the finite case (a negative lineability result)} We start by proving that $\mathcal{C}(L,\K)\setminus\overline{\mathcal{A}}$ is $\textup{ord}(\sim_{\mathcal{A}})$-lineable but not $(\textup{ord}(\sim_{\mathcal{A}})+1)$-lineable by assuming that $\sim_{\mathcal{A}}$ is finite and that $\mathcal{A}\subseteq\mathcal{C}(L,\K)$ is a self-adjoint algebra over $\K$ that vanishes nowhere on $L$. More precisely, we have the following result. 

\begin{theorem}\label{nonlineablereal}
Let $\mathcal{A}\subseteq\mathcal{C}(L,\K)$ be a self-adjoint algebra over $\K$ that vanishes nowhere on $L$ and such that $\sim_{\mathcal{A}}$ is finite. 
Then $\mathcal{C}(L,\K)\setminus\overline{\mathcal{A}}$ is $\textup{ord}(\sim_{\mathcal{A}})$-lineable but not $(\textup{ord}(\sim_{\mathcal{A}})+1)$-lineable.
\end{theorem}

\begin{proof}
We know that $\overline{\mathcal A}=\mathcal{C}_{\sim_{\mathcal A}}(L,\K)$ by Lemma~\ref{lemrealequi}. 
To see that it is not $(\textup{ord}(\sim_{\mathcal{A}})+1)$-lineable, note that 
    $$
    \text{codim}(\mathcal{C}_{\sim_{\mathcal{A}}}(L,\K))=\text{dim}(\mathcal{C}_L(\sim_{\mathcal{A}},\R))\leq \text{dim}(\K^{\text{ord}(\sim_{\mathcal{A}})}) = \text{ord}(\sim_{\mathcal{A}})
    $$
by Lemma~\ref{orderinjection} and the fact that $\text{ord}(\sim_{\mathcal A})$ is finite.
So let us now prove that it is indeed $\text{ord}(\sim_{\mathcal A})$-lineable.\\

Consider $R$ any set of representatives of $\sim_{\mathcal{A}}$. 
Then $R\cap\{x\in L:|[x]|>1\}$ is finite.
Let $R\cap\{x\in L:|[x]|>1\}=\{x_1,\ldots,x_m\}$ and, for each $1\leq i\leq m$, take $[x_i]\setminus\{x_i\}=\{x_i^1,\ldots,x_i^{n_i}\}$ which is also finite. 
Now, for each $1\leq j_i\leq n_i$ and $1\leq i\leq m$, we have that
    $$
    \{x_i^{j_i}\}\text{ and }\{x_1,\ldots,x_m\}\cup\{x_k^{\ell_k}:1\leq k\leq m,\ 1\leq \ell_k\leq n_k\}\setminus\{x_i^{j_i}\}
    $$ 
are disjoint closed subsets of $L$ and so, by Urysohn's lemma, there is an $f_i^{j_i}\in\mathcal{C}(L,[0,1])$ (where we identify $\mathcal{C}(L,[0,1])$ with $\{f\in\mathcal{C}(L,\mathbb K):f[L]\subseteq[0,1]\}$) such that
    $$
    f_{i}^{j_i}(x_i^{j_i})=1,\ f_i^{j_i}(x_k^{\ell_k})=0\text{ for every }x_k^{\ell_k}\neq x_i^{j_i}\text{ and }f_i^{j_k}(x_k)=0\text{ for every }1\leq k\leq m.
    $$
Now consider $V:=\textup{span}_{\K}\{f_i^{j_i}:1\leq i\leq m,\ 1\leq j_i\leq n_i\}$. 
Clearly, $\{0\}=\mathcal{C}_{\sim_{\mathcal{A}}}(L,\K)\cap V$ since $\sum\lambda_k^{\ell_k}f_k^{\ell_k}\in\mathcal{C}_{\sim_{\mathcal{A}}}(L,\K)$ implies that $0=\left(\sum \lambda_k^{\ell_k}f_k^{\ell_k}\right)(x_i)=\left(\sum \lambda_k^{\ell_k}f_k^{\ell_k}\right)(x_i^{j_i})=\lambda_i^{j_i}$. 
Also, they are linearly independent since $\sum\lambda_k^{\ell_k}f_k^{\ell_k}=0$ implies $0=\left(\sum \lambda_k^{\ell_k}f_k^{\ell_k}\right)(x_i^{j_i})=\lambda_i^{j_i}$. 
Hence, as $\text{dim}(V)=\sum_{i=1}^m n_i=\text{ord}(\sim_{\mathcal{A}})$, we conclude that $\mathcal{C}(L,\K)\setminus\overline{\mathcal{A}}$ is $\textup{ord}(\sim_{\mathcal{A}})$-lineable.
\end{proof}

\begin{remark}
Note that the degenerate case $\text{ord}(\sim_{\mathcal A})=0$ (i.e., $\sim_{\mathcal A}=D_K$ so that $\mathcal A$ separates points) is also contemplated in the latter proof giving the result $\overline{\mathcal A}=\mathcal{C}(K,\R)$, that is, $\mathcal{C}(K,\R)\setminus\overline{\mathcal A}=\varnothing$, where, clearly, $\varnothing$ is $0$-lineable but not $1$-lineable. 
As it is commonly done, $S=\{s_1,\ldots,s_n\}$ is implied to be the empty set $\varnothing$ whenever $n=0$ as this is just a bijection $s:n\to S$ where $n$ is a natural number under Neumann's set-theoretic definition $0=\varnothing$, $n+1=n\cup\{n\}$.
\end{remark}

As an immediate consequence of Theorem~\ref{nonlineablereal} and \cite[Theorem~2.1]{O}, we obtain the following corollary.
Recall that given a Banach lattice $V$, $M\subset V$ and a cardinal number $\kappa$, we say that $M$ is $\kappa$-latticeable if $M\cup \{0\}$ contains a sublattice $W$ of $V$ of dimension $\kappa$.
We refer the interested reader to \cite{BFMS,BFST,BL,CFST,FST,O2} for more information on latticeability.

\begin{corollary} \label{latticeable}
Let $\mathcal{A}\subseteq\mathcal{C}(L,\R)$ be an algebra over $\R$ that vanishes nowhere on $L$ and such that $\sim_{\mathcal{A}}$ is finite. 
Then $\mathcal{C}(L,\R)\setminus\overline{\mathcal{A}}$ is $\textup{ord}(\sim_{\mathcal{A}})$-latticeable but not $(\textup{ord}(\sim_{\mathcal{A}})+1)$-latticeable.
\end{corollary}

\subsection{On the infinite case} In this section, we focus on studying the lineability properties in the failure of the Stone-Weierstrass theorem in the infinite case, that is, the lineability properties of $\mathcal C(L,\K)\setminus \overline{\mathcal A}$ whenever $\sim_{\mathcal A}$ is infinite. We have the following result. 

\begin{proposition}\label{linnumber}
Let $\mathcal{A}\subseteq\mathcal{C}(L,\K)$ be a self-adjoint algebra over $\K$ that vanishes nowhere on $L$ and such that $\sim_{\mathcal{A}}$ is infinite. 
Then $\mathfrak{c} \leq \textup{codim}(\overline{\mathcal{A}}) \leq 2^{\textup{ord}(\sim_{\mathcal{A}})}$.
In particular, $\mathcal{C}(L,\K) \setminus \overline{\mathcal{A}}$ is $\mathfrak{c}$-lineable but not $(2^{\textup{ord}(\sim_{\mathcal{A}})})^+$-lineable. 
\end{proposition}

\begin{proof}
By Lemma~\ref{lemrealequi}, we already have that $\overline{\mathcal A}=\mathcal{C}_{\sim_{\mathcal A}}(L,\K)$. To see that $\textup{codim}(\overline{\mathcal{A}}) \leq 2^{\textup{ord}(\sim_{\mathcal{A}})}$, let us notice that Lemma~\ref{orderinjection} yields that
    \begin{align*}
        \text{codim}(\mathcal{C}_{\sim_{\mathcal{A}}}(L,\K)) & = \text{dim}(\mathcal{C}_L(\sim_{\mathcal{A}},\K)) \leq \text{dim}(\K^{\text{ord}(\sim_{\mathcal A})}) \\
        & \leq |\K^{\text{ord}(\sim_{\mathcal A})}|=\mathfrak{c}^{\text{ord}(\sim_{\mathcal A})}=2^{\aleph_0\text{ord}(\sim_{\mathcal A})}=2^{\text{ord}(\sim_{\mathcal A})}
    \end{align*}
as $\text{ord}(\sim_{\mathcal A})\geq\aleph_0$. To prove the remaining inequality, take any set of representatives $R\subseteq L$ of $\sim_{\mathcal A}$ and $\xi_R$ the associated map of $R$. 
Then, the graph of $\xi_R\circ\pi$ restricted to $L\setminus R$ is $\Gamma(\xi_R\circ\pi|_{L\setminus R})=((L\setminus R) \times R)\cap\sim_{\mathcal A}\subseteq\sim_{\mathcal A}$ with infinite cardinality as $|\Gamma(\xi_R\circ\pi|_{L\setminus R})|=|L\setminus R|=\text{ord}(\sim_{\mathcal A})$. 
Then, for any $n\geq1$, we can find $\{z_1,\ldots,z_n\}\subseteq((L\setminus R)\times R)\cap\sim_{\mathcal A}$. 
Let us denote each $z_i$ by $z_i=(x_i,y_i)$ and note that, since $\{x_1,\ldots,x_n\}\subseteq L\setminus R$ and $\{y_1,\ldots,y_n\}\subseteq R$, we have by Urysohn's lemma that for every $1\leq i\leq n$, there is an $f_i\in\mathcal{C}(L,[0,1])$ (where we identify $\mathcal{C}(L,[0,1])$ with $\{f\in\mathcal{C}(L,\mathbb K):f[L]\subseteq[0,1]\}$) such that
$$
f_i(x_j)=\delta_{ij} \text{ and } f_i(y_j)=0.
$$ 
Then, the functions $\Delta_{\sim_{\mathcal A}}(f_i)\in\mathcal{C}_L(\sim_{\mathcal A},\K)$ are linearly independent since $\sum_j\lambda_j\Delta_{\sim_{\mathcal A}}(f_j)\equiv 0$ implies that $0=\left(\sum\lambda_j\Delta_{\sim_{\mathcal A}}(f_j)\right)(z_i)=\lambda_i$ as 
    $$
    \Delta_{\sim_{\mathcal A}}(f_j)(z_i)=f_j(x_i)-f_j(y_i)=\delta_{ij}-0=\delta_{ij}.
    $$
This means that $\mathcal{C}_L(\sim_{\mathcal A},\K)$ has dimension $\geq n$. 
And, since $n\geq 1$ was arbitrary, it follows that $\mathcal{C}_L(\sim_{\mathcal A},\K)$ is an infinite-dimensional Banach space and so $\mathcal{C}(L,\K)\setminus\overline{\mathcal{A}}$ is $\mathfrak{c}$-lineable.
\end{proof}


\begin{remark}
For infinite $\sim_\mathcal{A}$ we also have the property that, if $\mathcal{C}(L,\K) \setminus \overline{\mathcal{A}}$ is $\kappa$-lineable, then it is also $\kappa^{\aleph_0}$-lineable. 
Indeed, if $\kappa \leq \textup{codim}(\overline{\mathcal{A}})$, then
	$$
	\kappa^{\aleph_0}\leq\textup{codim}(\overline{\mathcal{A}})^{\aleph_0}= \textup{codim}(\overline{\mathcal{A}}),
	$$
where we have used the fact that $\textup{codim}(\overline{\mathcal{A}})$ is the infinite dimension of a Banach space and, for any infinite-dimensional Banach space $V$, we have $|V|^{\aleph_0}=|V|$ (see \cite[Lemma~2]{Kruse}).
\end{remark}

\begin{remark}
    Similarly to Corollary~\ref{latticeable}, assuming that $L\subset \R^n$ for some $n\in \mathbb N$, it yields from Proposition~\ref{linnumber} and \cite[Theorem~1.2]{O2} that $\mathcal{C}(L,\R)\setminus\overline{\mathcal{A}}$ is completely latticeable, where $\mathcal{A}\subseteq\mathcal{C}(L,\R)$ is an algebra over $\R$ that vanishes nowhere on $L$.
    Recall that $M$ a subset of Banach lattice $V$ is completely latticeable if there exists a closed infinite-dimensional sublattice $W$ such that $W\setminus \{0\} \subset M$.
    However, unlike Corollary~\ref{latticeable}, we were unable to provide the largest possible infinite cardinal number $\kappa$ such that there is a closed $\kappa$-dimensional sublattice but not a closed $\kappa^+$-dimensional sublattice inside of $\mathcal{C}(L,\R)\setminus\overline{\mathcal{A}}$, even when $L$ is contained in the finite dimensional space $\mathbb R^n$.
\end{remark}

The upper bound on the codimension of $\overline{\mathcal A}$ in Proposition~\ref{linnumber} can be quite useful in terms of lineability, but the lower bound not so much as $\mathfrak c$ does not depend on $\sim_{\mathcal A}$. 
So let us try to find some better bounds by introducing a new cardinal function. Let us recall the following well-known result.

\begin{lemma}\label{complexification}
Given a topological space $X$, we have that $\textup{dim}_\R(\mathcal{C}(X,\R))=\textup{dim}_\C(\mathcal{C}(X,\C))$.
\end{lemma}


With Lemma~\ref{complexification} in mind, we set
\begin{equation*}
    sw(L):=\text{dim}(\mathcal{C}(L,\K))
\end{equation*}
and call it as the {\it Stone-Weierstrass character of $L$}. In the case that $L$ is finite, we have that $sw(L)=|L|$. Otherwise, we have that $sw(L)=|\mathcal{C}(L,\K)|$. With this notion, let us prove first a relation between the topological weight of $L$, that we denote by $w(L)$, and the Stone-Weierstrass character of $L$.

\begin{lemma}\label{bounds}
The following chain of inequalities holds true:
\begin{equation*} 
w(L)\leq sw(L)\leq w(L)^{\aleph_0}.
\end{equation*}
\end{lemma}

\begin{proof}
As we have said in the paragraph prior to Lemma~\ref{bounds}, if $L$ is finite, then $sw(L)=|L|=w(L)$ and we are done. Otherwise, it is well-known that $\text{dens}(\mathcal{C}(L,\K))=w(L)$. 
So consider $D\subseteq\mathcal{C}(L,\K)$ to be dense with cardinality $w(L)$. 
Clearly, $sw(L)=|\mathcal{C}(L,\K)|\geq |D|=w(L)$.
Now, since $\mathcal{C}(L,\K)$ is a metric space, we have that for any $f\in \mathcal{C}(L,\K)$, there exists a sequence $(f_n)_{n=1}^\infty \subseteq D$ that converges pointwise to $f$. 
This defines an injection from $\mathcal{C}(L,\K)$ into $D^\N$ that maps $f$ to $(f_n)_{n=1}^\infty$, which means that $|\mathcal{C}(L,\K)|\leq|D|^{\aleph_0}=w(L)^{\aleph_0}$.
\end{proof}

As a consequence of Lemma \ref{bounds}, we can express the Stone-Weiestrass character as follows.

\begin{proposition}
Let $L$ be a compact space, then
	$$
	sw(L)=\begin{cases}|L| & \text{if $L$ is finite}, \\ w(L)^{\aleph_0}& \text{if $L$ is infinite.}\end{cases}.
	$$
\end{proposition}

\begin{proof}
The finite case is clear. For the infinite case, denote $\kappa:=sw(L)=|\mathcal{C}(L,\K)|$. Since $\mathcal{C}(L,\K)$ is an infinite-dimensional Banach space, we have $\kappa^{\aleph_0}=\kappa$. 
Then, the previous inequality immediately implies $\kappa\leq w(L)^{\aleph_0}\leq\kappa^{\aleph_0}=\kappa$, concluding the proof.
\end{proof}

The notion of the Stone-Weierstrass character on itself can provide valuable information regarding lineability properties of $\mathcal{C}(L,\K)\setminus\overline{\mathcal{A}}$ even in the more general sets of the form $\mathcal{C}(L,\K)\setminus\mathcal{C}_\rho(L,\K)$ as shown by the following theorem.

\begin{theorem}\label{thmcharacter}
Let $R$ be a set of representatives of $\rho$ and let $S\subseteq(L\times R)\cap\rho$ be any closed subspace. Then, we have the following exact sequence
$$\mathcal{C}_L(\rho,\K)\xrightarrow{\textup{res}_S^\rho\vert_{\mathcal{C}_L(\rho,\K)}}\mathcal{C}(S,\K)\xrightarrow{\textup{res}_{S\cap D_L}^S}\mathcal{C}(S\cap D_L,\K).$$
In particular, we have
$$\textup{codim}(\mathcal{C}_\rho(L,\K))=\textup{dim}(\textup{Ker}(\textup{res}_S^\rho)\cap\mathcal{C}_L(\rho,\K))+\textup{dim}(\textup{Ker}(\textup{res}_{S\cap D_L}^S)).$$
\end{theorem}

\begin{proof}
We just have to prove that $\text{res}_S^\rho[\mathcal{C}_L(\rho,\K)]=\{f\in\mathcal{C}(S,\K):f[S\cap D_L]=0\}$ by double inclusion. First, for every $f\in\mathcal{C}_L(\rho,\K)$, we have $f[D_L]=0$ and so $f\vert_S[S\cap D_L]=0$. To see the other inclusion, let $f\in\mathcal{C}(S,\K)$ be such that $f[S\cap D_L]=0$. 
Now note that 
\begin{equation*} 
    S_X:=\{x\in L\ : (x,y)\in S \text{ for some } y \in L\}\subseteq L
\end{equation*} 
is closed since the projection map $(x,y)\mapsto x$ is closed by compactness. 
Now, if we consider the associated map $\xi_R$, we have that $S\subseteq (L\times R)\cap\rho=\Gamma(\xi_R\circ\pi|_L)$ implies $S=\Gamma(\xi_R\circ\pi|_{S_X})$. 
And, since $S=\Gamma(\xi_R\circ\pi|_{S_X})$ is closed and 
\begin{equation*} 
    (\xi_R\circ\pi)[S_X]=S_Y:=\{y\in L\ : (x,y)\in S \text{ for some } x \in L \}
\end{equation*} 
is compact and Hausdorff, it follows that $\xi_R\circ\pi|_{S_X}:S_X\to S_Y$ is continuous. 
Now consider $g:S_X\cup S_Y\to\K$ defined as
\begin{equation*} 
    g(x)=
    \begin{cases}
    f(x,\xi_R([x])) & \text{if } x\in S_X,\\ 
    0 & \text{if }x\in S_Y,
    \end{cases}
\end{equation*}
which is well-defined since given $x\in S_X\cap S_Y$ we have that $ x\in R$ and, therefore, $f(x,\xi_R([x]))=f(x,x)=0$. 
As $g\vert_{S_X}$ and $g\vert_{S_Y}$ are continuous, $g\in\mathcal{C}(S_X\cup S_Y,\K)$ where $S_X\cup S_Y\subseteq L$ is compact and Hausdorff. Then, by Tietze's extension theorem, there exists $\overline{g} \in\mathcal{C}(L,\K)$ such that $\overline{g} \vert_{S_x\cup S_y}=g$. So, if we now consider $\overline{f}:=\Delta_\rho(\overline{g})\in\mathcal{C}_L(\rho,\K)$, we have that
\begin{equation*} 
    \overline{f}\vert_{S}(x,y)=\overline{g}\vert_{S_X}(x)-\overline{g}\vert_{S_Y}(y)=g\vert_{S_X}(x)-0=f(x,\xi_R([x]))=f(x,y)
\end{equation*} 
for every $(x,y)\in S$ and so $f=\overline{f}\vert_S\in\text{res}_S^\rho[\mathcal{C}_L(\rho,\K)]$ as needed. For the last part, note that by kernel-image dimensionality we have $\text{codim}(\mathcal{C}_\rho(L,\K))=\text{dim}(\mathcal{C}_L(\rho,\K))$ with
\begin{align*}
\text{dim}(\mathcal{C}_L(\rho,\K))&=\text{dim}(\text{Ker}(\text{res}_S^\rho\vert_{\mathcal{C}_L(\rho,\K)}))+\text{dim}(\text{Im}(\text{res}_S^\rho\vert_{\mathcal{C}_L(\rho,\K)}))\\
&= \text{dim}(\text{Ker}(\text{res}_S^\rho)\cap\mathcal{C}_L(\rho,\K))+\text{dim}(\text{Ker}(\text{res}_{S\cap D_L}^S))
\end{align*}
\end{proof}

\begin{proposition} \label{prop-sw(S)}
Let $R$ be a set of representatives of $\rho$ such that there is an infinite closed subset $S\subseteq(L\times R)\cap\rho$ with $sw(S)>sw(S\cap D_L)$. 
Then, denoting
$$\kappa_0:=|\{f\in\mathcal{C}_L(\rho,\K):f\vert_{S}\equiv0\}|$$
we have that $\mathcal{C}(L,\K)\setminus\mathcal{C}_\rho(L,\K)$ is $(sw(S)+\kappa_0)$-lineable but not $(sw(S)+\kappa_0)^+$-lineable. 
\end{proposition}

\begin{proof}
We just need to prove that $\text{codim}(\mathcal{C}_\rho(L,\K))=sw(S)+\kappa_0$. To do this, note that considering the linear surjection $\text{res}_{S\cap D_L}^S:\mathcal{C}(S,\K)\to\mathcal{C}(S\cap D_L,\K)$, we have
$$\underbrace{\text{dim}(\mathcal{C}(S,\K))}_{sw(S)}=\underbrace{\text{dim}(\mathcal{C}(S\cap D_L,\K))}_{sw(S\cap D_L)}+\text{dim}(\text{Ker}(\text{res}_{S\cap D_L}^S))=\text{dim}(\text{Ker}(\text{res}_{S\cap D_L}^S))$$
by kernel-image dimensionality and the fact that $sw(S)>sw(S\cap D_L)$ where $sw(S)$ is infinite. Hence, as $\text{dim}(\text{Ker}(\text{res}_{S\cap D_L}^S))=sw(S)\geq\mathfrak c$ and $\kappa_0=\textup{dim}(\textup{Ker}(\textup{res}_S^\rho)\cap\mathcal{C}_L(\rho,\K))+\mathfrak{c}$, it follows by the Theorem~\ref{thmcharacter} that
\begin{align*}
\text{codim}(\mathcal{C}_\rho(L,\K))&=sw(S)+\textup{dim}(\textup{Ker}(\textup{res}_S^\rho)\cap\mathcal{C}_L(\rho,\K))\\
&=sw(S)+\textup{dim}(\textup{Ker}(\textup{res}_S^\rho)\cap\mathcal{C}_L(\rho,\K))+\mathfrak{c}\\
&=sw(S)+\kappa_0
\end{align*} 
The last part follows immediately from $(sw(S)+\kappa_0)$-lineability.
\end{proof}

This last proposition immediately implies the following result.

\begin{proposition}\label{supcardinal}
Let $\mathcal{A}\subseteq\mathcal{C}(L,\K)$ be a self-adjoint algebra over $\K$ that vanishes nowhere on $L$. If $\sim_{\mathcal A}$ is infinite, then $\mathcal{C}(L,\K)\setminus\overline{\mathcal A}$ is $(\sup_{x\in L}w([x]))^{\aleph_0}$-lineable.
\end{proposition}

\begin{proof}
Consider any set of representatives $R$ of $\sim_\mathcal{A}$ and let $[r]\in L/\sim_{\mathcal A}$ be an arbitrary infinite equivalence class with $r\in R$. If $[r]$ is finite, $w([r])^{\aleph_0}=\mathfrak c$ and so $\mathcal{C}(L,\K)\setminus\overline{\mathcal A}$ is $w([r])^{\aleph_0}$-lineable. Otherwise, if $[r]$ is infinite, since $\sim_\mathcal{A}$ is closed, we know that $[r]\subseteq L$ is closed and so is $[r]\times\{r\}\subseteq(L\times R)\cap\sim_\mathcal{A}$. Moreover, $([r]\times\{r\})\cap\sim_\mathcal{A}=\{(r,r)\}$ and so $sw([r]\times\{r\})\geq\mathfrak{c}>1=sw(([r]\times\{r\})\cap D_L)$. Thus, by the  Theorem~\ref{thmcharacter}, $\mathcal{C}(L,\K)\setminus\overline{\mathcal A}$ is $sw([r])$-lineable where $sw([r])=w([r])^{\aleph_0}$. In turn, $\mathcal{C}(L,\K)\setminus\overline{\mathcal A}$ must be $\sup_{x\in L}(w([x])^{\aleph_0})$-lineable and so, since we have the following chain of cardinal inequalities 
$$\left(\sup_{x\in L}w([x])\right)^{\aleph_0}\leq\left(\sup_{x\in L}(w([x])^{\aleph_0})\right)^{\aleph_0}\leq\left(\left(\sup_{x\in L}w([x])\right)^{\aleph_0}\right)^{\aleph_0}=\left(\sup_{x\in L}w([x])\right)^{\aleph_0}$$
which shows that $\left(\sup_{x\in L}(w([x])^{\aleph_0})\right)^{\aleph_0}=\left(\sup_{x\in L}w([x])\right)^{\aleph_0}$, we can finally conclude that the set $\mathcal{C}(L,\K)\setminus\overline{\mathcal A}$ must also be $\left(\sup_{x\in L}w([x])\right)^{\aleph_0}$-lineable, as desired.
\end{proof}

%
%
%

The thing with Theorem \ref{thmcharacter} is that, although usually it is not too hard to compute the dimensionality of $\text{Ker}(\text{res}_{S\cap D_L})$, the computation of the dimensionality of $\text{Ker}(\text{res}_{S})\cap C_L(\rho,\K)$ may be as hard as computing the dimensionality of $C_L(\rho,\K)$ directly. For this reason, in the general case, Theorem \ref{thmcharacter} is quite useful to find lower bounds of $\kappa$-lineability (positive results) but not so much to find upper bounds of $\kappa$-lineability (negative results).\\

However, if we restrict ourselves to the following particular case, we can obtain from Theorem \ref{thmcharacter} a much straightforward computation. It turns out that computing the codimensionality of $\mathcal{C}_\rho(L,\K)$ reduces to computing the dimensionality of a space of continuous functions that vanish on a closed subset of its domain.  

\vspace{0.2cm}

For the rest of this paper, given $R$ a set of representatives of $\rho$, we will denote
	$$
	L':=\{x\in L:|[x]|>1\}.
	$$

\begin{proposition}\label{prpcont}
Let $R$ be a set of representatives of $\rho$ such that $\xi_R\circ\pi\vert_{\overline{L'}}$ is continuous.
Then, denoting $R':=\overline{L'}\cap R$ and $\kappa_0:=\textup{dim}(\textup{Ker}(\textup{res}_{\overline{R'}}^{\overline{L'}}))$, we have that $\mathcal{C}(L,\K)\setminus\mathcal{C}_\rho(L,\K)$ is $\kappa_0$-lineable but not $\kappa_0^+$-lineable.
\end{proposition}

\begin{proof}
Since $\xi_R\circ\pi\vert_{\overline{L'}}$ is continuous with Haudorff codomain, we know that $S:=\Gamma(\xi_R\circ\pi\vert_{\overline{L'}})$ is a closed subset of $(L\times R)\cap\rho$. Thus, we can apply Theorem \ref{thmcharacter} as
$$\textup{codim}(\mathcal{C}_\rho(L,\K))=\textup{dim}(\textup{Ker}(\textup{res}_S^\rho)\cap\mathcal{C}_L(\rho,\K))+\textup{dim}(\textup{Ker}(\textup{res}_{S\cap D_L}^S))$$
Now note that, given $\Delta_\rho(f)\in\mathcal{C}_L(\rho,\K)$ (where $f\in\mathcal{C}(L,\K)$), the condition $\Delta_\rho(f)\vert_{S}\equiv0$ is equivalent to $f(x)-f(\xi_R\circ\pi(x))=0$ for all $x\in \overline{L'}$.
Also, $f(x)-f(\xi_R\circ\pi(x))=0$ is also true for $x\not\in \overline{L'}$. Indeed, if $x\not\in \overline{L'}$, then $x\not\in L'$ and so $|[x]|=1$, which implies that $x=\xi_R\circ\pi(x)$.
Hence, it follows that $\Delta_\rho(f)\equiv0$ and so $\textup{Ker}(\textup{res}_S^\rho)\cap\mathcal{C}_L(\rho,\K)=\{0\}$. Thus,
$$\textup{codim}(\mathcal{C}_\rho(L,\K))=\textup{dim}(\textup{Ker}(\textup{res}_{S\cap D_L}^S))$$

It only remains to prove that $\textup{Ker}(\textup{res}_{S\cap D_L}^S)$ and $\textup{Ker}(\textup{res}_{\overline{R'}}^{\overline{L'}})$ are linearly isomorphic. 
To do this, just consider the map $\Lambda:\textup{Ker}(\textup{res}_{S\cap D_L}^S)\to\textup{Ker}(\textup{res}_{\overline{R'}}^{\overline{L'}})$ given by
	$$
	\Lambda(f)(x)=f(x,\xi_R\circ\pi(x)).
	$$
It is well defined by the continuity of $\xi_R\circ\pi\vert_{\overline{L'}}$ and the fact that $S=\Gamma(\xi_R\circ\pi\vert_{\overline{L'}})$.
Indeed, we have $S\cap D_L=\Gamma(\xi_R\circ\pi\vert_{\overline{L'}})\cap D_L=D_{R'}$ since $\text{Fix}(\xi_R\circ\pi\vert_{\overline{L'}})=\overline{L'}\cap R=R'$ so $\Lambda(f)\vert_{R'}\equiv0$ and, by continuity, $\Lambda(f)\vert_{\overline{R'}}\equiv0$.
It is clear that $\Lambda$ is linear. 
Also, it is bijective as it can be easily checked that its inverse $\Lambda^{-1}$ is given by $\Lambda^{-1}(f)(x,y)=f(x)$ which is well-defined since 
$$
S_X:=\{x\in L : \exists(x,y)\in S=\Gamma(\xi_R\circ\pi\vert_{\overline{L'}})\}=\overline{L'}.
$$
This shows the linear isomorphism between $\textup{Ker}(\textup{res}_{S\cap D_L}^S)$ and $\textup{Ker}(\textup{res}_{\overline{R'}}^{\overline{L'}})$, which yields $\textup{codim}(\mathcal{C}_\rho(L,\K))=\textup{dim}(\textup{Ker}(\textup{res}_{\overline{R'}}^{\overline{L'}}))$.
\end{proof}

Proposition~\ref{prpcont} may not seem that helpful at first glance since one would expect the continuity of $\xi_R\circ\pi\vert_{\overline{L'}}$ to be rather uncommon. But, in fact, it gives us three immediate corollaries that are extremely useful in many cases.\\

First, for $\rho$ to be of infinite order, it must be the case that: (A) there are infinitely many classes $[x]\in L/\rho$ such that $|[x]|>1$ or (B) there are finitely many classes $[x]\in L/\rho$ such that $|[x]|>1$ one of which is of infinite cardinality. It turns out that Proposition \ref{prpcont} exhaustively answers the $\kappa$-lineability problem for the latter case.

\begin{corollary} \label{cor:prpcont}
Let $\mathcal{A}\subseteq\mathcal{C}(L,\K)$ be a self-adjoint algebra over $\K$ that vanishes nowhere on $L$. If $\sim_{\mathcal A}$ is infinite and $\{[x]\in L/\sim_{\mathcal{A}}:|[x]|>1\}$ is finite, then  $\mathcal{C}(L,\K)\setminus\overline{\mathcal A}$ is $sw(L')$-lineable but not $sw(L')^+$-lineable.
\end{corollary}

\begin{proof}
Consider any set of representatives $R$ of $\sim_{\mathcal{A}}$. First, $R'=R\cap L'=\{r_1,\ldots,r_n\}$ has to be finite. Since $\sim_{\mathcal{A}}$ is closed, we also know that $L'=\bigcup_{i=1}^n[r_i]$ has to be closed. Now consider the map $\xi_R\circ\pi\vert_{L'}:L'\to R$ and note that it has to be continuous since $L'=\bigcup_{i=1}^n[r_i]$ and for each closed set $[r_i]$ the restriction $\xi_R\circ\pi\vert_{[r_i]}$ is constant and hence continuous. This means that we can apply Proposition \ref{prpcont} to $\overline{\mathcal{A}}=\mathcal{C}_{\sim_{\mathcal{A}}}(L,\K)$ as
$$
\text{codim}(\overline{\mathcal{A}})=\text{dim}(\textup{Ker}(\textup{res}_{L'\cap R}^{L'})).
$$
Moreover, since $\mathcal{C}(L',\K)$ is infinite-dimensional and $\textup{Im}(\textup{res}_{L'\cap R}^{L'}))=\mathcal{C}(L'\cap R,\K)$ is finite-dimensional, it follows by kernel-image dimensionality that 
$$sw(L')=\text{dim}(\mathcal{C}(L',\K))=\text{dim}(\textup{Ker}(\textup{res}_{L'\cap R}^{L'}))+\text{dim}(\textup{Im}(\textup{res}_{L'\cap R}^{L'}))=\text{dim}(\textup{Ker}(\textup{res}_{L'\cap R}^{L'}))$$
which, in turn, implies $\text{codim}(\overline{\mathcal{A}})=sw(L')$ as desired.
\end{proof}

In order to provide the second corollary, we need some additional notation and definitions (take into account the notions we have provided in Subsection \ref{lineabilitynotation}). Fávaro, Pellegrino, Raposo Jr. and Ribeiro introduced the concept of $(\alpha,\beta)$-dense-lineability in \cite{FPRR} similarly to $(\alpha,\beta)$-spaceability. We say that $M$ is {\it $(\alpha,\beta)$-dense-lineable} if $M$ is $\alpha$-dense-lineable and for every $\alpha$-dimensional vector subspace $V_\alpha$ of $V$ with $V_\alpha \subseteq M\cup \{0\}$, there is a dense $\beta$-dimensional subspace $V_\beta$ of $V$ such that $V_\alpha \subseteq V_\beta \subseteq M\cup \{0\}$ holds true. Also, $(\alpha_1,\beta)$-dense-lineability does not imply $(\alpha_2,\beta)$-dense-lineability and viceversa. Also in \cite{FPRR}, the concept of pointwise $\alpha$-dense-lineability was established. We say that $M$ is {\it pointwise $\alpha$-dense-lineable} if for every $x\in M$ there is a dense $\alpha$-dimensional subspace $V_\alpha$ of $V$ such that $x\in V_\alpha \subseteq A\cup \{0\}$. It is obvious that pointwise $\beta$-dense-lineability implies $(1,\beta)$-dense-lineability, but the converse does not hold in general. Finally, we say that $M$ is {\it pointwise maximal-dense-lineable} if $M$ is pointwise $\dim(V)$-dense-lineable. 
Now, we are ready to provide the following result.

\begin{corollary}\label{corsup}
	Let $\mathcal{A}\subseteq\mathcal{C}(L,\K)$ be a self-adjoint algebra over $\K$ that vanishes nowhere on $L$. 
	If $\sim_{\mathcal A}$ is infinite and $\{[x]\in L/\sim_{\mathcal{A}}:|[x]|>1\}$ is finite, then $\mathcal{C}(L,\K)\setminus\overline{\mathcal A}$ is
	\begin{itemize}
		\item[(i)] $(\alpha,\mathfrak c)$-spaceable if and only if $\alpha < \aleph_0$, and
		
		\item[(ii)] pointwise $\mathfrak c$-spaceable.
	\end{itemize}
	Moreover, if $w(L)\leq sw(L')$, then $\mathcal{C}(L,\K)\setminus\overline{\mathcal{A}}$ is
	\begin{itemize}
		\item[(iii)] $(\alpha,\beta)$-dense-lineable for every $\alpha < sw(L')$ and $\max\{\alpha,w(L)\}\leq \beta \leq sw(L')$, and
		
		\item[(iv)] pointwise $sw(L')$-dense-lineable.
	\end{itemize}
\end{corollary}

The third corollary of Proposition \ref{prpcont} is the following which has to do with retractions.

\begin{corollary}\label{corretract}
Let $S\subseteq L$ be a closed subset of $L$ and $T$ a retract of $S$ given by the retraction $r:S\to T$. Denote
	$$
	\mathcal{C}_r(L,\K)=:\{f\in\mathcal{C}(L,\K) : f(x)=f(r(x)) \text{ for all } x\in S \},
	$$
Then $\mathcal{C}_r(L,\K)$ is a closed subspace of $\mathcal{C}(L,\K)$ and
	$$
	\textup{codim}(\mathcal C_r(L,\K))=\textup{dim}(\{f\in\mathcal{C}(S,\K):f\vert_T\equiv0\})=:\kappa_0,
	$$
that is, $\mathcal{C}(L,\K)\setminus\mathcal{C}_r(L,\K)$ is $\kappa_0$-lineable but not $\kappa_0^+$-lineable. \end{corollary}

\begin{proof}
Consider the relation $\sim_r$ defined on $L$ as
	$$
	x\sim_r y\iff x=y \text{ or } r(x)=r(y) \text{ provided that } x,y\in S.
	$$ 
Note that the function $r\times r:S\times S\to S\times S$ given by $(x,y)\mapsto (r(x),r(y))$ is continuous and so $\sim_r=D_L\cup (r\times r)^{-1}[D_T]$ is closed. This means that $\mathcal{C}_r(L,\K)=\mathcal{C}_{\sim_r}(L,\K)$ and we can apply Proposition~\ref{prpcont}. 
We have that $L'=\{x\in L:|[x]_{\sim_r}|>1\}=(S\setminus T)\cup r[S\setminus T]$ and the natural choice for the set of representatives is $R:=T$. Now,
$$\overline{L'}=\overline{S\setminus T}\cup \overline{r[S\setminus T]}=\overline{S\setminus T}\cup \overline{r\left[\overline{S\setminus T}\right]}=\overline{S\setminus T}\cup r\left[\overline{S\setminus T}\right]=S\setminus T^{\mathrm{o}}\cup r[S\setminus T^{\mathrm{o}}]$$
and so, since $T$ is closed, $R':=\overline{L'}\cap R=\overline{L'}\cap T=\partial T\cup r[S\setminus T^{\mathrm{o}}]$ which is closed. Now, from Proposition~\ref{prpcont}, we deduce that
$$
\textup{codim}(\mathcal C_r(L,\K))=\textup{dim}(\{f\in\mathcal{C}(\overline{L'},\K):f\vert_{R'} \equiv 0\})
$$
which is almost what we wanted. To obtain the correct expression, consider the map $H:\text{Ker}(\text{res}^S_T)\to\text{Ker}(\text{res}^{\overline{L'}}_{R'})=\text{Ker}(\text{res}^{S\cap\overline{L'}}_{T\cap\overline{L'}})$ given by $f\mapsto f\vert_{\overline{L'}}$ (i.e., $H=\left(\text{res}^S_{\overline{L'}}\right)\vert_{\text{Ker}(\text{res}^S_T)}$). It is clearly well-defined and linear. 
To see that it is injective, note that, if $f\vert_T\equiv0$, then $f\vert_{\overline{L'}}\equiv0$ implies that $f\equiv0$ since $T\cup\overline{L'}=S$ and so $\text{Ker}(H)=\{0\}$. To see that it is surjective, for a given $g\in\text{Ker}(\text{res}^{\overline{L'}}_{R'})$, extend it as $\overline{g}:S\to\R$ with
$$\overline{g}(x)=\begin{cases}g(x)&\text{if }x\in\overline{L'}\\ 0&\text{if }x\in T\end{cases}$$
which is well defined since $\overline{L'}\cap T=R'$ and $g\vert_{R'}\equiv0$. Moreover, since it is continuous in $\overline{L'}$ and in $T$ with $\overline{L'}\cup T=S$, $\overline{g}$ is continous and, by construction, $\overline{g}\in\text{Ker}(\text{res}^{S}_{T})$. So, since $\overline{g}\vert_{\overline{L'}}=g$, this proves the surjectivity of $H$. Overall,  $\text{Ker}(\text{res}^S_T)$ is linearly isomorphic to $\text{Ker}(\text{res}^{\overline{L'}}_{R'})$ and so
$$\textup{codim}(\mathcal C_r(L,\K))=\textup{dim}(\{f\in\mathcal{C}(S,\K):f\vert_T\equiv0\})$$
as we wanted to show.
\end{proof}

Let us now display the power of Corollary~\ref{corretract} with two examples.

\begin{example}
Let $\kappa\geq \aleph_0$.
Consider the Tychonoff cube $[0,1]^\kappa$ and the subset of $\mathcal{C}([0,1]^\kappa,\K)$ given by $W=\{f\in\mathcal{C}([0,1]^\kappa,\K):f(0,\cdot)=f(1,\cdot)\}$. 
Let us try to study the $\kappa$-lineability of $\mathcal{C}([0,1]^\kappa,\K)\setminus W$. 

Denote the projection map that gives the first entry of points in $[0,1]^\kappa$ as $\pi_0:[0,1]^\kappa\to[0,1]$.
Now denote $F_0=\pi_0^{-1}[\{0\}]$ and $F_1=\pi_0^{-1}[\{1\}]$ which may be seen as two ``opposite faces'' of Tychonoff's cube. Then, $S:=F_0\cup F_1\subseteq[0,1]^\kappa$ is a closed subset and $T:=F_0$ is a retract of $S$ given by the retraction $r:S\to T$ with
$$r(x)(\alpha)=\begin{cases}0 & \text{if }\alpha=0\\ x(\alpha)&\text{if }\alpha>0\end{cases}$$
Now note that $W=\mathcal{C}_r([0,1]^\kappa,\K)$ which means that we can apply the Corollary~\ref{corretract} as
$$\text{codim}(W)=\text{dim}(\{f\in\mathcal{C}(F_0\cup F_1,\K):f\vert_{F_0} \equiv0\})$$
and, since $\{F_0,F_1\}$ is a disconnection of $F_0\cup F_1$ and $F_1$ is homeomorphic to $[0,1]^\kappa$, we conclude that $\text{dim}(\{f\in\mathcal{C}(F_0\cup F_1,\K):f\vert_{F_0} \equiv0\})=\text{dim}(\mathcal{C}([0,1]^\kappa,\K))$.
But also
$$\text{dim}(\mathcal{C}([0,1]^\kappa,\K))=sw([0,1]^\kappa)=w([0,1]^\kappa)^{\aleph_0}=\kappa^{\aleph_0}$$
which means that $\mathcal{C}([0,1]^\kappa,\K)\setminus W$ is maximal-lineable.
Hence, by \cite[Corollary~4.4]{FPRR}, the set $\mathcal{C}([0,1]^\kappa,\K)\setminus W$ is pointwise maximal-dense-lineable.
\end{example}

\begin{example}
Consider again the Tychonoff cube $[0,1]^\kappa$ with $\kappa\geq \aleph_0$ and the subset of $\mathcal{C}([0,1]^\kappa,\K)$ given by 
	$$
	U=\{f\in\mathcal{C}([0,1]^\kappa,\K) : f(t_1,\cdot)=f(t_2,\cdot) \text{ for every } t_1,t_2\in [0,1]\}
	$$
That is, $U$ is the space of functions that are constant with respect to the first entry. Now, using the same notation as in the previous example, note that $F_0$ is a retract of $[0,1]^\kappa$ given by $r:[0,1]^\kappa\to F_0$ with
$$r(x)(\alpha)=\begin{cases}0 & \text{if }\alpha=0\\ x(\alpha)&\text{if }\alpha>0\end{cases}$$
(the previous retraction is the restriction of this one to $F_0\cup F_1$). Then we have that $U=\mathcal{C}_r([0,1]^\kappa,\K)$ and so, by Corollary \ref{corretract}, it follows that
	$$
	\text{codim}(U)=\text{dim}(\{f\in\mathcal{C}([0,1]^\kappa,\K):f\vert_{F_0}\equiv 0\})=\text{dim}(\text{Ker}(\text{res}_{F_0}^{I^\kappa})),
	$$
with $I=[0,1]$.
Now, if we denote $Q:=\pi_0^{-1}[[1/2,1]]=\{x\in[0,1]^\kappa:x(0)\in[1/2,1]\}$ which is homeomorphic to $[0,1]^\kappa$, we have that the map $H:\text{Ker}(\text{res}_{F_0}^{I^\kappa})\to\mathcal{C}(Q,\K)$ given by $f\mapsto f\vert_Q$ (so that $H=\text{res}_{Q}^{I^\kappa}\vert_{\mathcal{C}(\text{Ker}(\text{res}_{F_0}^{I^\kappa}),\K)}$) is a well defined linear function. 
Let us now show that $H$ is surjective. Given $g\in\mathcal{C}(Q,\K)$, define $\overline{g}:F_0\cup Q\to\K$ given by
	$$
	\overline{g}(x)=
	\begin{cases}
		g(x) & \text{if }x\in Q,\\ 
		0 & \text{if }x\in F_0.
	\end{cases}
	$$
Since $Q$ and $F_0$ are disjoint and closed, $\overline{g}$ is well-defined and continuous. Then, by Tietze's extension Theorem, there is some $G\in\mathcal{C}([0,1]^\kappa,\K)$ such that $G\vert_{Q\cup F_0}=\overline{g}$. 
In particular, $G\vert_{F_0}\equiv0$, so that $G\in \text{Ker}(\text{res}_{F_0}^{I^\kappa})$ and $H(G)=G\vert_{Q}=g$ which shows the surjectivity of $H$. That means that $\text{dim}(\text{Ker}(\text{res}_{F_0}^{I^\kappa}))\geq\text{dim}(\mathcal{C}(Q,\K))=\text{dim}(\mathcal{C}([0,1]^\kappa,\K))=sw([0,1]^\kappa)=\kappa^{\aleph_0}$ and 
$$\kappa^{\aleph_0}\leq\text{dim}(\text{Ker}(\text{res}_{F_0}^{I^\kappa}))=\text{codim}(U)\leq\kappa^{\aleph_0}$$
So, again, and by \cite[Corollary~4.4]{FPRR}, the set $\mathcal{C}([0,1]^\kappa,\K)\setminus U$ is pointwise maximal-dense-lineable.
\end{example}

A theorem of Miljutin \cite{M} (see also \cite{P}) states that if $L_1$ and $L_2$ are uncountable metric compact sets, then $\mathcal C(L_1,\mathbb K)$ and $\mathcal C(L_2,\mathbb K)$ are isomorphic.
In particular, for any $L$ infinite compact metric set and $\kappa\geq \aleph_0$, $\mathcal C(L^\kappa,\mathbb K)$ is isomorphic to $\mathcal C([0,1]^\kappa,\mathbb K)$.
Thus, in order to study, for instance, the pointwise maximal-dense-lineability of $C(L^\kappa,\mathbb K)\setminus V$, where $V$ is a subspace of $C(L^\kappa,\mathbb K)$, it is enough to analyze $C(L^\kappa,\mathbb K)\setminus W$, where $W$ is isomorphic to a subspace of $\mathcal C([0,1]^\kappa,\mathbb K)$.
The latter is summarized in the following trivial remark. \\

\begin{remark}\label{miljutin1}
	Let $L$ be an infinite metric compact set and $\kappa \geq \aleph_0$.
	If $W$ is a subspace of $\mathcal C([0,1]^\kappa,\mathbb K)$, then the following assertions are equivalent:
	\begin{itemize}
		\item[\textup{(i)}]$\mathcal C([0,1]^\kappa,\mathbb K)\setminus W$ is pointwise maximal-dense-lineable.
		
		\item[\textup{(ii)}] $\mathcal C(L^\kappa,\mathbb K)\setminus \Phi(W)$ is pointwise maximal-dense-lineable for some isomorphism $\Phi : \mathcal C([0,1]^\kappa,\mathbb K) \to \mathcal C(L^\kappa,\mathbb K)$.
		
		\item[\textup{(iii)}] $\mathcal C(L^\kappa,\mathbb K)\setminus \Phi(W)$ is pointwise maximal-dense-lineable for any isomorphism $\Phi : \mathcal C([0,1]^\kappa,\mathbb K) \to \mathcal C(L^\kappa,\mathbb K)$.
	\end{itemize}.
\end{remark}

\subsection{On the complex plane} \label{complex-plane}


Let us now focus on the complex plane $\mathbb{K}=\C$ with $L=\partial D$ which is the boundary of the open unit disk $D:=\{z\in\C:|z|<1\}$. We also define
    $$
    \textup{Pol}(\partial D):=\{f\vert_{\partial D}:f\in\C[z]\}
    $$
and
    $$
    \textup{Hol}(\partial D):=\{f\vert_{\partial D}:f\in\mathcal{C}(\overline D,\C)\text{ is holomorphic on }D\},
    $$
where $\C[z]$ denotes the set of complex polynomials in one variable.
\\

We will now work with the notion of strong alebrability. If $V$ is a (not necessarily unital) algebra, $M\subseteq V$ and $\kappa$ is a cardinal number, then we say that $M$ is {\it strongly $\kappa$-algebrable} if $M\cup \{0\}$ contains a $\kappa$-generated free algebra.
Recall that if $V$ is commutative, then a set $B\subseteq V$ is a
generating set of some free algebra contained in $M$ if and only if for any $n\in \mathbb N$, any nonzero polynomial $P$ in $n$ variables without constant term and any distinct $f_1,\ldots,f_n \in B$, we have $P(f_1,\ldots,f_n)\neq 0$ and $P(f_1,\ldots,f_n)\in M$.

\vspace{0.2cm}

Let $U$ be a domain in $\C$, that is, a nonempty connected open subset of $\C$.
Endow the space of holomorphic functions from $U$ to $\C$, denoted by $\text{Hol}(U)$, with the topology of uniform convergence on compacta (also known as the compact-open topology).
This makes $\text{Hol}(U)$ into a complete metrizable locally convex
topological vector space.
A domain $U\subseteq \C$ is a \textit{Runge domain} if for every $f\in \text{Hol}(U)$, every compact $K\subseteq U$ and every $\varepsilon >0$, there exists a $P\in \C[z]$ such that $|f(z)-P(z)|<\varepsilon$ for every $z\in K$.
If the latter property is not satisfied, then we say that $U$ is a \textit{non-Runge domain}.
All open balls in $\C$ are Runge domain, but, for instance, $\C\setminus \{0\}$ is a non-Runge domain since $f(z)=1/z$ cannot be approximated by polynomials in this topology.
Let us denote $\text{Pol}(U):=\{P|_U:P\in \C[z] \}$.
In \cite{BCLS}, Bernal-González, Calderón-Moreno, López-Salazar and Seoane-Sepúlveda showed that $\text{Hol}(U)\setminus \overline{\text{Pol}(U)}$ is dense-lineable, spaceable, closely algebrable and  strongly $\mathfrak c$-algebrable, assuming that $U\subseteq \C$ is a non-Runge domain.
This line of study continued in \cite{B}, where Ballisco answered a question posed in \cite{BCLS} and framed this type of problem in a more general context.

In this section, we contribute to this ongoing research in the following way.
Instead of considering a Runge domain $U$, we take $\partial D$ which is compact subset of $\C$.
By doing so, now the space $\text{Hol}(\partial D)$ endowed with the topology of uniform convergence on compacta is metrizable and the induced metric is derived form the supremum norm.
Furthermore, as a consequence of Morera's theorem, $\text{Hol}(\partial D)$ is a closed subspace of $\mathcal{C}(\partial D,\C)$.
In this way, we prove that $\mathcal{C}(\partial D,\C)\setminus\overline{\textup{Pol}(\partial D)}\cup \{0\}$ is not only maximal-spaceable and strongly $\mathfrak c$-algebrable, but also $\mathcal{C}(\partial D,\C)\setminus\overline{\textup{Pol}(\partial D)}$ contains an isometric copy of $\text{Hol}(\partial D)$, among other lineability properties.

\vspace{0.2cm}

Let us define the map 
    \begin{align*}
        \varphi:\mathcal{C}(\partial D,\C) & \to\mathcal{C}(\partial D,\C) \\
        f(z) & \mapsto \bar zf(\bar z),
    \end{align*}
and endow $\mathcal{C}(\partial D,\C)$ with the supremum norm (which turns $\mathcal{C}(\partial D,\C)$ into a Banach algebra).
It is easy to prove that $\varphi$ is an involutive linear isometry and, therefore, $\varphi$ is a homeomorphism and a linear automorphism on $\mathcal{C}(\partial D,\C)$.
Since $\textup{Hol}(\partial D)$ is a closed subspace of $\mathcal{C}(\partial D,\C)$, it yields that $\varphi[\textup{Hol}(\partial D)]$ is isomorphically isometric to $\textup{Hol}(\partial D)$ and, furthermore, the following lemma shows that $\varphi[\textup{Hol}(\partial D)]\subseteq (\mathcal{C}(\partial D,\C)\setminus\overline{\textup{Pol}(\partial D)})\cup \{0\}$.

\begin{lemma}\label{lemmacomplexbase}
The set $(\mathcal{C}(\partial D,\C)\setminus\overline{\textup{Pol}(\partial D)})\cup \{0\}$ contains an isometric copy of $\textup{Hol}(\partial D)$.
In particular, $\varphi[\textup{Hol}(\partial D)]\setminus \{0\}\subseteq \mathcal{C}(\partial D,\C)\setminus\overline{\textup{Pol}(\partial D)}$.
\end{lemma}

\begin{proof}
It is enough to show that $\varphi[\textup{Hol}(\partial D)]\setminus \{0\} \subseteq \mathcal{C}(\partial D,\C)\setminus\overline{\textup{Pol}(\partial D)}$.
Assume first that $f(0)\neq 0$.
Observe that, since $\bar z=1/z$ for any $z\in\partial D$, we have via the Cauchy integral formula
    $$
    \frac{1}{2\pi i}\oint_{\partial D}\bar zf(\bar z)dz=\frac{1}{2\pi i}\oint_{\partial D}\frac{1}{z}f\left(\frac{1}{z}\right)dz=\frac{1}{2\pi i}\oint_{\partial D}\frac{f\left(w\right)}{w}dw=f(0),
    $$
where we have made the substitution $w=1/z$. 
Now, given $p\in\text{Pol}(\partial D)$, we have
    \begin{align*}
        \norm{\varphi(f)-p}_\infty & = \sup_{z\in\partial D}|\bar zf(\bar z)-p(z)| \\
        & \geq\frac{1}{2\pi}\oint_{\partial D}\vert\bar zf(\bar z)-p(z)\vert |dz| \\
        & \geq\left\vert\frac{1}{2\pi i}\oint_{\partial D}(\bar zf(\bar z)-p(z))dz\right\vert=|f(0)|,
    \end{align*}
since $\oint_{\partial D} p(z)dz=0$. Hence, we have $d(\varphi(f),\text{Pol}(\partial D))\geq|f(0)| >0$, as needed.

To finish the proof, assume that $f(0)=0$.
Then, there is an $n \in \N$ and a $g\in\textup{Hol}(\partial D)$ with $g(0) \neq 0$ such that $f(z)=z^ng(z)$.
Now, by way of contradiction, assume that $\varphi(f)\in\overline{\text{Pol}(\partial D)}$. 
Then, there would exist $(p_k(z))_{k=1}^\infty \subseteq\text{Pol}(\partial D)$ converging uniformly to $\varphi(f)(z)$ on $\partial D$. 
But note that this would imply that $(z^np_k(z))_{k=1}^\infty \subseteq\text{Pol}(\partial D)$ converges to 
    $$
    z^n\varphi(f)(z)=z^n\bar zf(\bar z)=z^n\bar z\bar z^ng(\bar z)=\bar zg(\bar z)=\varphi(g)(z),
    $$
which is impossible since $g(0)\neq0$. 
\end{proof}

\begin{theorem} 
The set $\mathcal{C}(\partial D,\C)\setminus\overline{\textup{Pol}(\partial D)}$ is maximal-spaceable.
\end{theorem}

\begin{proof}
Lemma~\ref{lemmacomplexbase} guarantees that $\mathcal{C}(\partial D,\C)\setminus\overline{\textup{Pol}(\partial D)}$ is $\mathfrak c$-spaceable.
The maximal-spaceability follows from the fact that every infinite-dimensional Banach space has dimension at least $\mathfrak c$, and $\mathcal{C}(\partial D,\C)$ has dimension $\mathfrak c$ since it is a separable infinite-dimensional Banach space.
\end{proof}


The results from \cite{FPRR} imply the following corollary.

\begin{corollary}
The set $\mathcal{C}(\partial D,\C)\setminus\overline{\textup{Pol}(\partial D)}$ is
\begin{itemize}
    \item[(i)] $(\alpha,\mathfrak c)$-spaceable if and only if $\alpha < \aleph_0$,

    \item[(ii)] pointwise $\mathfrak c$-spaceable,

    \item[(iii)] $(\alpha,\beta)$-dense-lineable for every $\alpha < \mathfrak c$ and $\max\{\alpha,\aleph_0\}\leq \beta \leq \mathfrak c$, and

    \item[(iv)] pointwise maximal-dense-lineable.
\end{itemize}
\end{corollary}

To end this section, we will study the strong algebrability of $\mathcal{C}(\partial D,\C)\setminus\overline{\textup{Pol}(\partial D)}$.

\begin{theorem} 
    The set $\mathcal{C}(\partial D,\C)\setminus\overline{\textup{Pol}(\partial D)}$ is strongly $\mathfrak c$-algebrable.
\end{theorem}

\begin{proof}
Let $\mathcal H$ be a Hamel basis of $\R$ over $\Q$. 
    (Recall that $|\mathcal H|=\mathfrak c$.)
    
    It is well-known (and not difficult to prove) that the functions
    $\{ z^a e^{bz} : a,b \in \C\}$ are linearly independent over $\C$ and, therefore, the functions in $\{ \bar{z}^a e^{b\bar{z}} : a,b \in \C\}$ are linearly independent. Hence, observe that the functions in 
        $$
        \mathcal B := \{ \bar{z} e^{h\bar z} : h \in \mathcal H \}
        $$
    are algebraically independent.
    Indeed, if $P$ is a nonzero complex polynomial with $n$ variables and without constant term, we have that
        \begin{equation}\label{equ:1}
            P(\bar{z}e^{h_1\bar z},\ldots,\bar{z}e^{h_n\bar z}) = \sum_{j=1}^k a_j \bar{z}^{\sum_{i=1}^n m_{i,j}} e^{\sum_{i=1}^n m_{i,j}h_i \bar{z}}
        \end{equation}
    where $k\in \mathbb N$, $h_1,\ldots,h_n$ are distinct, $a_j\in \C$ for each $1\leq j\leq m$, $(m_{1,j},\ldots,m_{n,j})\in (\N\cup \{0\})^n$ with $\sum_{i=1}^n m_{i,j}\geq 1$ for all $1\leq j\leq k$ and $(m_{1,j_1},\ldots,m_{n,j_1})\neq (m_{1,j_2},\ldots,m_{n,j_2})$ provided that $j_1\neq j_2$.
    Then, $P(\bar{z}e^{h_1\bar z},\ldots, \bar{z}e^{h_n\bar z})\equiv 0$ only when $a_j=0$ for all $j$ by the $\C$-linear independence of $\{ \bar{z}^a e^{b\bar{z}} : a,b \in \C\}$ since the numbers $\sum_{i=1}^n m_{i,1}h_i,\ldots,\sum_{i=1}^n m_{i,k}h_i$ are mutually distinct by the $\Q$-linear independence of $\mathcal H$.

  It remains to prove that $P(\bar{z}e^{h_1\bar z},\ldots, \bar{z}e^{h_n\bar z})\in(\mathcal{C}(\partial D,\C)\setminus\overline{\text{Pol}(\partial D)})\cup\{0\}$. For this, it will suffice to show that $P(\bar{z}e^{h_1\bar z},\ldots, \bar{z}e^{h_n\bar z})\in\varphi[\text{Hol}(\partial D)]$ or, since $\varphi$ is an involution, that $\varphi(P(\bar{z}e^{h_1\bar z},\ldots, \bar{z}e^{h_n\bar z}))\in\text{Hol}(\partial D)$. For this, simply note that
    \begin{align*}
        &\varphi\left(\sum_{j=1}^k a_j \bar{z}^{\sum_{i=1}^n m_{i,j}} e^{\sum_{i=1}^n m_{i,j}h_i \bar{z}}\right)=\bar{z}\left(\sum_{j=1}^k a_j z^{\sum_{i=1}^n m_{i,j}} e^{\sum_{i=1}^n m_{i,j}h_i z}\right)\\
        &=z^{-1}\left(\sum_{j=1}^k a_j z^{\sum_{i=1}^n m_{i,j}} e^{\sum_{i=1}^n m_{i,j}h_i z}\right)=\sum_{j=1}^k a_j z^{\sum_{i=1}^n m_{i,j}-1} e^{\sum_{i=1}^n m_{i,j}h_i z}
    \end{align*}
    where we have used the fact that $\bar{z}=z^{-1}$ in $\partial D$. This last function is clearly holomorphic since $\sum_{i=1}^n m_{i,j}\geq1$ for every $1\leq j\leq n$.
\end{proof}

\vspace{0.3cm}
\noindent
{\bf Acknowledgments}: The authors thank Tommaso Russo for his valuable discussions on various aspects of this manuscript and for providing helpful references.

\vspace{0.5cm}

\noindent
{\bf Funding information}: Sheldon Dantas was supported by The Spanish AEI Project PID2019 - 106529GB - I00 / AEI / 10.13039/501100011033,
the Generalitat Valenciana project CIGE/2022/97 and 
he was also funded by the grant PID2021-122126NB-C33 funded by 
MICIU/AEI/10.13039/501100011033 and by ERDF/EU.


\begin{thebibliography}{FaHaMo}

\bibitem{ABRR} \textsc{G.~Araújo, A. Barbosa, A. Raposo Jr., and G. Ribeiro}, \textit{On the spaceability of the set of functions in the Lebesgue space $L_p$ which are not in $L_q$}, Bull Braz Math Soc, New Series {\bf 54}, 44 (2023). 

\bibitem{ABPS} \textsc{R.~Aron, L.~Bernal-González, D. M.~Pellegrino, and J. B.~Seoane-Sepúlveda}, \textit{Lineability: the search
for linearity in mathematics}, Monographs and Research Notes in Mathematics, CRC Press, Boca
Raton, FL, 2016.




\bibitem{B} \textsc{E. Ballico}, \textit{Non-integral elements in holomorphic algebras}, Linear Algebra Appl., {\bf 520} (2017), pp. 77--79. DOI \hyperlink{https://doi.org/10.1016/j.laa.2017.01.015}{j.laa.2017.01.015}.
.

\bibitem{BCLS} \textsc{L. Bernal-González, M. C. Calderón-Moreno, J. López-Salazar, and J .B. Seoane-Sepúlveda}, \textit{Holomorphic functions on non-Runge domains and related problems}, Linear Algebra Appl., {\bf 511} (2016), pp. 226--236. DOI \hyperlink{http://dx.doi.org/10.1016/j.laa.2016.09.021}{10.1016/j.laa.2016.09.021}.

\bibitem{BFMS} \textsc{L. Bernal-González, J. Fernández-Sánchez, M. E. Mart\'inez-G\'omez, and J .B. Seoane-Sepúlveda}, \textit{Banach spaces and Banach lattices of singular functions}, Studia Math., {\bf 260} (2021), no. 2, 167--193. DOI \hyperlink{https://doi.org/10.4064/sm200419-7-9}{10.4064/sm200419-7-9}.

\bibitem{BFST} \textsc{L. Bernal-González, J. Fernández-Sánchez, J. B. Seoane-Sepúlveda, and W. Trutschnig}, \textit{Highly tempering infinite matrices II: From divergence to convergence via Toeplitz-Silverman matrices}, Rev. R. Acad. Cienc. Exactas F\'is. Nat. Ser. A Mat. RACSAM, {\bf 114} (2020), no. 4, Paper No. 202, 10. DOI \hyperlink{https://doi.org/10.1007/s13398-020-00934-z}{10.1007/s13398-020-00934-z}.

\bibitem{BPS} \textsc{L.~Bernal-González, D. M.~Pellegrino, and J. B.~Seoane-Sepúlveda}, \textit{Linear subsets of nonlinear sets in topological vector spaces}, Bull. Amer. Math. Soc. (N.S.) {\bf 51} (2014), no. 1, 71--130. DOI \href{https://doi.org/10.1090/S0273-0979-2013-01421-6}{10.1090/S0273-0979-2013-01421-6}.


\bibitem{BFPS} \textsc{G.~Botelho, V. V.~Fávaro, D.~Pellegrino and J. B.~Seoane-Sepúlveda}, \textit{$L_p[0,1]\setminus \bigcup_{q>p} L_q[0,1]$ is spaceable for every $p > 0$}, Linear Algebra Appl. {\bf 436} (2012), 2963--2965.

\bibitem{BL} \textsc{G. Botelho, and J. L. P. Luiz}, \textit{Complete latticeability in vector-valued sequence spaces}, Math. Nachr., {\bf 296} (2023), no. 2, 523--533.

@article {MR4553610,
    AUTHOR = {},
     TITLE = {},
   JOURNAL = {},
  FJOURNAL = {Mathematische Nachrichten},
    VOLUME = {296},
      YEAR = {2023},
    NUMBER = {2},
     PAGES = {},
      ISSN = {0025-584X,1522-2616},
   MRCLASS = {46B87 (46B42 46B45 46G25)},
  MRNUMBER = {4553610},
}


\bibitem{CFST} \textsc{J. Carmona Tapia, J. Fern\'andez-S\'anchez, J. B. Seoane-Sep\'ulveda, and W. Trutschnig}, \textit{Lineability, spaceability, and latticeability of subsets of $C([0,1])$ and Sobolev spaces}, Rev. R. Acad. Cienc. Exactas F\'is. Nat. Ser. A Mat. RACSAM, {\bf 116} (2022), no. 3, Paper No. 113, 20. DOI \href{https://doi.org/10.1007/s13398-022-01256-y}{10.1007/s13398-022-01256-y}.

\bibitem{C} \textsc{K. Ciesielski}, \textit{Set theory for the working mathematician}, London Mathematical Society Student Texts, vol. 39, Cambridge University Press, Cambridge, 1997. DOI \href{https://doi.org/10.1017/CBO9781139173131}{10.1017/CBO9781139173131}.

\bibitem{FPRR} \textsc{V. V.~Fávaro, D.~Pellegrino, A.~Raposo Jr., and G.~Ribeiro}, \textit{General criteria for a strong notion of lineability}, Proc. Amer. Math. Soc. {\bf 152} (2024), 941--954. DOI \hyperlink{https://doi.org/10.1090/proc/16608}{10.1090/proc/16608}.


\bibitem{FGK} \textsc{V. P.~Fonf, V. I.~Gurariy, and M. I.~Kadets}, \textit{An infinite dimensional subspace of $C[0, 1]$ consisting of
nowhere differentiable functions}, C. R. Acad. Bulgare Sci. {\bf 52} (1999), no. 11--12, 13--16.

\bibitem{FST} \textsc{J. Fern\'andez-S\'anchez, J. B. Seoane-Sep\'ulveda, and W. Trutschnig}, \textit{Lineability, algebrability, and sequences of random variables}, Math. Nachr. \textbf{295} (2022), no. 5, 861--875. DOI \href{https://doi.org/10.1002/mana.202000102}{10.1002/mana.202000102}. 

\bibitem{FPT}
  \textsc{V.V. Fávaro, D. Pellegrino, and D. Tomaz}, 
 \textit{Lineability and Spaceability: A New Approach}, Bull. Braz. Math. Soc., New Series 51, (2020), 27–46.

\bibitem{GQ} \textsc{V. I. Gurariy and L. Quarta}, \textit{On lineability of sets of continuous functions}, J. Math. Anal. Appl. \textbf{294} (2004), no. 1, 62--72. DOI \href{https://doi.org/10.1016/j.jmaa.2004.01.036}{10.1016/j.jmaa.2004.01.036}. 

\bibitem{HK} \textsc{J.E.~Harta and K.~Kunenb}, \textit{Compacta and their function spaces}, Topology and its Applications {\bf 195} (2015) 119--142

\bibitem{J} \textsc{T.~Jech}, \textit{Set theory}, third edition, Springer--Verlag, New York, 2002.

 
\bibitem{KT} \textsc{D. Kitson, and R.M. Timoney}, \textit{Operator ranges and spaceability}, J. Math. Anal. Appl. {\bf 378} (2011), 680--686. DOI \hyperlink{https://doi.org/10.1016/j.jmaa.2010.12.061}{10.1016/j.jmaa.2010.12.061}.

\bibitem{Kruse} \textsc{A. H. Kruse}, \textit{Badly incomplete normed linear spaces}, Math. Z. {\bf 83} (1964), 314--320. DOI \href{https://doi.org/10.1007/BF01111164}{10.1007/BF01111164}.

\bibitem{K} \textsc{K. Kunen}, \textit{The complex Stone-Weierstrass property}, Fund. Math. {\bf 182} (2004), no. 2, 151--167.

\bibitem{LRS} \textsc{P.~Leonetti, T.~Russo, and J.~Somaglia}, \emph{Dense lineability and spaceability in certain subsets of $\ell_{\infty}$}, Bull. London Math. Soc., Volume 55, Issue 5 p. 2283--2303. DOI \href{https://doi.org/10.1112/blms.12858}{10.1112/blms.12858}.

\bibitem{LM} \textsc{B.~Levine, and D.~Milman}, \textit{On linear sets in space $C$ consisting of functions of bounded variation}, Russian, with English summary. Commun. Inst. Sci. Math. Méc. Univ. Kharkoff [Zapiski Inst. Mat. Mech.] (4) {\bf 16}, 102--105 (1940).

\bibitem{M} \textsc{A. A. Miljutin}, \textit{Isomorphisms of spaces of continuous functions on compacts of power continuum}, Tieoria Func. (Kharkov), 2 (1966), 150--156 (Russian).

\bibitem{O} \textsc{T. A. Oikhberg}, \textit{A note on latticeability and algebrability}, J. Math. Anal. Appl. {\bf 434} (2016), no. 1, 523--537. DOI \href{https://doi.org/10.1016/j.jmaa.2015.09.025}{10.1016/j.jmaa.2015.09.025}. 

\bibitem{O2} \textsc{T. A. Oikhberg}, \textit{Large sublattices in subsets of Banach lattices}, Arch. Math. (Basel) {\bf 109} (2017), no. 3, 245--253. DOI \href{https://doi.org/10.1007/s00013-017-1070-z}{10.1007/s00013-017-1070-z}.

\bibitem{P} \textsc{A. Pelczy\'{n}ski}, \textit{On $C(S)$ subspaces of separable Banach spaces}, Studia Math. 31 (1968), 513--522.

\bibitem{Rm} \textsc{M. Rmoutil}, \emph{Norm-attaining functionals need not contain 2-dimensional subspaces}, J. Funct. Anal. {\bf 272} (2017), no. 3, 918--928. DOI \href{https://doi.org/10.1016/j.jfa.2016.10.028}{10.1016/j.jfa.2016.10.028}.

\bibitem{Rudin} \textsc{W.~Rudin}, {\it Subalgebras of spaces of continuous functions}, ibid. {\bf 7} (1956), 825--830


\bibitem{SS} \textsc{J. B.~Seoane-Sepúlveda}, \textit{Chaos and lineability of pathological phenomena in analysis}, Thesis (Ph.D.), Kent State University, ProQuest LLC, Ann Arbor, MI, pp. 139 (2006).

\bibitem{S} \textsc{K. R. Stromberg}, \textit{Introduction to classical real analysis}, 
Wadsworth Internat. Math. Ser., Wadsworth International, Belmont, CA, 1981. ix+575 pp.

\bibitem{W} \textsc{A. Wilansky}, \textit{Semi-Fredholm maps of FK spaces}, Math. Z. {\bf 144} (1975) 9--12. DOI \hyperlink{https://doi.org/10.1007/BF01214402}{10.1007/BF01214402}.

\end{thebibliography}
 \end{document}